\documentclass{article}

\usepackage{a4}
\usepackage{paralist}
\usepackage{graphicx, psfrag}
\usepackage{pgfplotstable}
\usepackage{subcaption}

\title{Mass-lumping discretization and solvers\\ for distributed elliptic optimal control problems} 
\author{Ulrich~Langer\footnote{Institute of Numerical Mathematics,
		Johannes Kepler University Linz, 
		and Johann Radon Institute for Computational and Applied Mathematics 
		of the Austrian Academy of Sciences,
		Altenberger Stra{\ss}e 69, 4040 Linz,
		Austria, Email: ulanger@numa.uni-linz.ac.at},
	\;
	Richard~L\"oscher\footnote{Institut f\"{u}r Angewandte Mathematik,
		Technische Universit\"{a}t Graz, Steyrergasse 30, 8010 Graz, Austria,
		Email: loescher@math.tugraz.at}, 
	\; Olaf~Steinbach\footnote{Institut f\"{u}r Angewandte Mathematik,
		Technische Universit\"{a}t Graz, Steyrergasse 30, 8010 Graz, Austria,
		Email: o.steinbach@tugraz.at}, 
	\; Huidong~Yang\footnote{Faculty of Mathematics, University of Vienna,
		Oskar--Morgenstern--Platz 1, 1090 Wien, Austria, and Christian
		Doppler Laboratory for Mathematical Modeling and Simulation of Next
		Generations of Ultrasound Devices (MaMSi),
		Oskar--Morgenstern--Platz 1, 1090 Wien, Austria, Email: huidong.yang@univie.ac.at}
}  

\date{}

\usepackage{amsmath}
\usepackage{amsthm}
\usepackage{amsfonts}
\usepackage{booktabs}
\usepackage{graphicx}
\usepackage{diagbox}

\usepackage{multirow}
\usepackage{hyperref}

\usepackage[ulem=normalem,draft]{changes}

\newcommand{\norm}[1]{\|#1\|}
\newcommand{\skpr}[1]{\left\langle #1\right\rangle}

\newtheorem{remark}{Remark}
\providecommand{\keywords}[1]{\textbf{\textit{Keywords:}} #1}

\newtheorem{theorem}{Theorem}
\newtheorem{lemma}{Lemma}

\begin{document}
	
	\maketitle
%
%
\begin{abstract}
The purpose of this paper is to investigate the effects of the use of mass-lumping 
in the finite element discretization of the reduced first-order optimality
system arising from a standard tracking-type, distributed 
elliptic optimal control problem with $L_2$ regularization.
We show that mass-lumping will not affect the $L_2$ error 
between the desired state and the computed state,
but will lead to a Schur-complement system
that allows for a fast matrix-by-vector multiplication.
We show that the use of the Schur-Complement Preconditioned Conjugate Gradient method 
in a nested iteration setting leads to an asymptotically optimal solver
with respect to the complexity.\\[2ex]
%
%
\end{abstract}
%
%
%
%
\section{Introduction}
\label{LLSY:sec:Introduction}	
%
\noindent
We consider the following tracking-type, distributed elliptic optimal control problem with standard 
$L_2$ regularization: 
find the state $y_\varrho \in Y=H_0^1(\Omega)$ and the control $u_\varrho \in U=L_2(\Omega)$  
minimizing the cost functional
\begin{equation}
\label{LLSY:Eqn:CostFunctional}
J(y_\varrho,u_\varrho) := \frac{1}{2} \, \|y_\varrho - y_d\|_{L_2(\Omega)}^2
+ \frac{\varrho}{2} \, \| u_\varrho  \|_{L_2(\Omega)}^2,
\end{equation}
subject to (s.t.) the elliptic boundary value model problem 
%
\begin{equation}
\label{LLSY:Eqn:EllipticBVP_SF}
-\Delta y_\varrho = u_\varrho\text{ in }\Omega, \quad
y_\varrho=0\text{ on }\partial\Omega, 
\end{equation}	
for some given desired state (target) $y_d\in L_2(\Omega)$, 
and some regularization parameter $\varrho >0$,
where $\Omega \subset \mathbb{R}^d$, $d=1,2,3$, is a bounded Lipschitz domain 
with the boundary $\partial \Omega$.
We here use the standard notations for Lebesgue and Sobolev spaces. 
Since the state equation \eqref{LLSY:Eqn:EllipticBVP_SF} has a unique solution $y_\varrho \in Y$ 
for every given control $u_\varrho \in U$, the optimal control problem \eqref{LLSY:Eqn:CostFunctional}-\eqref{LLSY:Eqn:EllipticBVP_SF} 
has a unique solution $(y_\varrho,u_\varrho) \in Y \times U$ too; 
see, e.g., \cite{LLSY:Lions:1968a}, \cite{LLSY:HinzePinnauUlbrichUlbrich:2009Book},
or \cite{LLSY:Troeltzsch:2010a}.
Moreover, the state $y_\varrho$ obviously belongs to $H^\Delta(\Omega)=\{y \in H_0^1(\Omega): \Delta y \in L_2(\Omega)\}$,
and the solution operator $S$ mapping $u_\varrho$ to $y_\varrho$ 
(control-to-state map)
is an isomorphism between $L_2(\Omega)$
and $H^\Delta(\Omega)$.

The finite element (fe) discretization of the reduced  (after elimination of the control $u_\varrho$)  optimality system, 
which defines the solution to  the optimal control problem \eqref{LLSY:Eqn:CostFunctional}-\eqref{LLSY:Eqn:EllipticBVP_SF}, 
leads to the solution of a large-scale symmetric, but indefinite linear system of algebraic equations 
for defining the fe nodal adjoint state vector 
$\mathbf{p}_h \in \mathbb{R}^{n_h}$
and the fe nodal state vector $\mathbf{y}_h \in \mathbb{R}^{n_h}$ such that
\begin{equation}
\label{LLSY:Eqn:DiscreteReducedOptimalitySystem}
  \begin{pmatrix}
    \varrho^{-1} M_h & K_h\\
   K_h & -M_h
  \end{pmatrix}
  \begin{pmatrix}
    \mathbf{p}_h\\
    \mathbf{y}_h
  \end{pmatrix}
  =
  \begin{pmatrix}
    \mathbf{0}_h\\
    -\mathbf{y}_{dh}
  \end{pmatrix},
\end{equation}
where the stiffness matrix $K_h$ and the mass matrix $M_h$ 
are symmetric and positive definite (spd),  
$\mathbf{y}_{dh} \in \mathbb{R}^{n_h}$ is nothing but the fe 
load vector representing  the desired state $y_d$,
and $h$ denotes a suitable discretization parameter.
For fixed $\varrho$, discretization error estimates can be found, e.g.,  in 
\cite{LLSY:HinzePinnauUlbrichUlbrich:2009Book}.
There is
a huge number of
publications on efficient preconditioned iterative solvers 
for symmetric, but indefinite systems in general; see, e.g.,
the unified approach proposed in \cite{LLSY:Zulehner:2001MathComp},
the survey paper \cite{LLSY:BenziGolubLiesen:2005ActaNumer}, 
the review article \cite{LLSY:MardalWinther:2011NLA}, 
the books \cite{LLSY:ElmanSilvesterWathen:2005Book} and \cite{LLSY:BaiPan:2021Book},
the more recent papers 
\cite{LLSY:AxelssonKaratson:2020NumerMath,LLSY:Bai:2019IMANA,
LLSY:BaiBenzi:2017BIT,LLSY:Notay:2019SIMAX},
and the literature cited therein.
%
Special iterative solvers for discrete optimality systems such as
\eqref{LLSY:Eqn:DiscreteReducedOptimalitySystem}
should be not only robust with respect to (wrt) the mesh refinement quantified by the 
discretization parameter $h$ but also wrt the regularization parameter $\varrho$
that can be quite small depending on the cost that we are willing to pay.
Such kind of $h$ and $\varrho$ robust preconditioned iterative methods have been 
proposed and investigated in 
\cite{LLSY:AxelssonKaratson:2020NumerMath,
LLSY:BaiBenziChenWang:2013IMANA,
LLSY:PearsonWathen:2012NLA,
LLSY:SchoeberlZulehner:2007SIMAX,
LLSY:Zulehner:2011SIMAX};
see also
\cite{LLSY:AxelssonNeytchevaStroem:2018JNM,
LLSY:DravinsNeytcheva:2021-003-nc,
LLSY:PearsonStollWathen:2014NLA,
LLSY:SchielaUlbrich:2014SIOPT,
LLSY:StollWathen:2012NLA},
for handling control and state constraints,
and the references therein. 
Alternatively, we can use all-at-once multigrid methods to solve 
saddle-point problems such as
\eqref{LLSY:Eqn:DiscreteReducedOptimalitySystem} efficiently;
see, e.g.,  \cite{LLSY:SchulzWittum:2008CVS}
and the review paper \cite{LLSY:BorziSchulz:2009SIAMreview}.

In this paper, we are interested in the case $\varrho=h^4$ leading to asymptotically 
optimal balanced estimates of the $L_2$-error between the desired state $y_d$ 
and the computed finite element state $y_{\varrho h}$ that is related to 
$\mathbf{y}_h$ by the fe isomorphism; 
see \cite{LLSY:LangerLoescherSteinbachYang:2023}.
Asymptotically optimal preconditioned iterative solvers for the saddle-point system 
\eqref{LLSY:Eqn:DiscreteReducedOptimalitySystem}
were proposed in 
\cite{LLSY:LangerLoescherSteinbachYang:2023} and \cite{LLSY:LangerLoescherSteinbachYang:2023LSSC}
for constant and variable $L_2$ regularizations, respectively.
More precisely, it turns out that very cheap preconditioners 
for the MINRES and BP-CG can be constructed on the basis of simple diagonal 
approximations of the mass matrix $M_h$. 
Of course, we can further reduce the saddle-point system \eqref{LLSY:Eqn:DiscreteReducedOptimalitySystem}
to the Schur-Complement (SC) system
\begin{equation}
 \label{LLSY:Eqn:SchurComplementSystemFullM}
		(\varrho K_h M_h^{-1}K_h +M_h) \mathbf{y}_h = \mathbf{y}_{dh} 
\end{equation}
by eliminating the adjoint state $\mathbf{p}_h$. The system matrix is spd, 
and we would like to solve this system by means of the 
Preconditioned Conjugate Gradient (PCG) method. 
Although we can use very cheap diagonal matrices $D_h$ 
such as $\mbox{diag}(M_h)$ or the lumped mass matrix $\mbox{lump}(M_h)$ 
as preconditioners that are spectrally equivalent to the Schur 
complement $\varrho K_h M_h^{-1}K_h +M_h$ for $\varrho=h^4$ 
\cite{LLSY:LangerLoescherSteinbachYang:2023LSSC,LLSY:LangerLoescherSteinbachYang:2023},
we cannot simply replace the mass matrix $M_h^{-1}$ by the lumped mass matrix $(\mbox{lump}(M_h))^{-1}$ 
in the Schur complement without a precise analysis of the impact of this replacement to the discretization error.
In Section~\ref{LLSY:sec:MassLumpingAndErroAnalysis}, we just provide this analysis,
and show that, in the case of continuous, piecewise linear (Courant's) 
finite element spaces, the discretization error is not affected at all. 
This theoretical result is supported by our numerical results 
presented in Section~\ref{LLSY:sec:NumericalResults}.
Now we have to solve the mass-lumped 
SC system
%
\begin{equation}
 \label{LLSY:Eqn:SchurComplementSystemLumpedMass}
 (\varrho K_h (\mbox{lump}(M_h))^{-1}K_h +M_h) \hat{\mathbf{y}}_h
 = \mathbf{y}_{dh} 
\end{equation}
instead of the original 
SC system \eqref{LLSY:Eqn:SchurComplementSystemFullM}. 
Using the  diagonal preconditioner
$D_h = \mbox{lump}(M_h)$, we can now solve \eqref{LLSY:Eqn:SchurComplementSystemLumpedMass} 
in asymptotically optimal complexity for some fixed relative accuracy. 
In Section~\ref{LLSY:sec:NestedPCGIteration}, we show how we can use 
this SC-PCG in a nested iteration setting in order to compute a fe approximation to the desired state $y_d$,
which differs from $y_d$ in the $L_2$-norm in the order of the discretization error,
with asymptotically optimal complexity $O(n_h)$. 
These theoretical results are again quantitatively illustrated 
by numerical experiments in Section~\ref{LLSY:sec:NumericalResults}.

%
%
\section{Mass-Lumping and Error  Analysis}
\label{LLSY:sec:MassLumpingAndErroAnalysis}	
%
	
	The first-order optimality system, derived from 
	\eqref{LLSY:Eqn:CostFunctional}--\eqref{LLSY:Eqn:EllipticBVP_SF}, 
	is given by 
	the equations
	\begin{equation}
		 \label{LLSY:Eqn:optimality-system1}
		-\Delta y_\varrho = u_\varrho,\; -\Delta p_\varrho = y_\varrho-y_d,\;\mbox{and}\; 		
			 p_\varrho +\varrho u_\varrho=0 \quad \mbox{in}\; \Omega,	 
	\end{equation} 
	%
	with the boundary conditions 
	\begin{equation} 
	\label{LLSY:Eqn:optimality-system2}
		y_\varrho = 0 \;\mbox{and}\;p_\varrho=0 \quad \mbox{on}\; \partial \Omega.
	\end{equation}
	%
	%
	Eliminating the control $u_\varrho$, we arrive at the reduced first-order optimality 
	system, the variational form of which reads as follows: 
	find $(y_\varrho,p_\varrho)\in H_0^1(\Omega)\times H_0^1(\Omega)$ such that 
	%
	%
	%
    \begin{eqnarray} 
    \label{LLSY:Eqn:VF-optimality-system1}
        \frac{1}{\varrho}\skpr{p_\varrho,q}_{L_2(\Omega)} + \skpr{\nabla y_\varrho,\nabla q}_{L_2(\Omega)} 
            &=& 0, \;\; \forall q\in H_0^1(\Omega),\\ 
     \label{LLSY:Eqn:VF-optimality-system2}       
        -\skpr{\nabla p_\varrho,\nabla v}_{L_2(\Omega)} + \skpr{y_\varrho,v}_{L_2(\Omega)} 
            &=& \skpr{y_d,v}_{L_2(\Omega)},
			\;\; \forall v\in H_0^1(\Omega).
    \end{eqnarray}
	Introducing the variable $\tilde p_\varrho = \frac{1}{\sqrt{\varrho}}p_\varrho$, we further derive the scaled system
	%
	%
	\begin{eqnarray}
	\label{LLSY:Eqn:VF-optimality-system-scaled1}
	\frac{1}{\sqrt{\varrho}}\skpr{\tilde p_\varrho,q}_{L_2(\Omega)} + \skpr{\nabla y_\varrho,\nabla q}_{L_2(\Omega)} &=& 0, 
	\;\; \forall q\in H_0^1(\Omega),\\
	\label{LLSY:Eqn:VF-optimality-system-scaled2}
	-\skpr{\nabla \tilde p_\varrho,\nabla v}_{L_2(\Omega)} + \frac{1}{\sqrt{\varrho}}\skpr{y_\varrho,v}_{L_2(\Omega)} &=& \frac{1}{\sqrt{\varrho}}\skpr{y_d,v}_{L_2(\Omega)}, \;\; \forall v\in H_0^1(\Omega).
	\end{eqnarray}
	For simplicity, we assume from now on
	that $\Omega\subset \mathbb{R}^d$ is polygonally ($d=2$) or polyhedrally ($d=3$) bounded. 
	Let $\mathcal{T}_h=\{\tau_e\}_{e=1}^{N_h}$ be an admissible,
      globally  
    quasi-uniform and shape-regular decomposition of $\Omega$ into simplicies $\tau_e$, with the mesh-size $h_e=|\tau_e|^{1/d}$, 
    such that 
	$\overline{\Omega}=\bigcup_{e=1}^{N_h}\overline{\tau}_e$. 
	Let $S_h^1(\mathcal{T}_h)=\text{span}\{\varphi_j^h\}_{j=1}^{\overline
          n_h}$ denote the space of piecewise linear, globally continuous
        functions
        spanned by the Lagrange basis functions $\varphi_j^h$ (hat functions), which fulfil the equations
	\begin{align}\label{eq:properties-phijh}
		\sum_{j=1}^{\overline{n}_h} \varphi_j^h(x) =1
		\;\; \forall x\in \Omega,
		\quad \text{and}\quad \varphi_j^h(x_i) = \delta_{i,j}\text{ for each node }x_i,\, i=1,\ldots,\overline{n}_h. 
	\end{align} 
	Further, we define $V_{h}:=S_h^1(\mathcal{T}_h)\cap H_0^1(\Omega)=\text{span}\{\varphi_j^h\}_{j=1}^{n_h}$, 
	where we assume that the ordering of the basis functions is such that
        the indices $j=1,\ldots,n_h$ correspond to vertices $x_j\in \Omega$ and
        $j=n_h+1,\ldots,\overline{n}_h$ corresponds to the vertices on the
        boundary, $x_j\in\partial \Omega$. 
	We refer to the books
 \cite{LLST:Braess:2007a,LLSY:ErnGuermond:2004,LLSY:Steinbach:2008Monograph}    
 for more details on standard finite element discretizations
 of elliptic PDEs.
	
	A conforming discretization of 
	\eqref{LLSY:Eqn:VF-optimality-system-scaled1}-\eqref{LLSY:Eqn:VF-optimality-system-scaled2}
	is then to find $(y_{\varrho h},\tilde p_{\varrho h})\in V_h\times V_h$ such that
	%
	%
	\begin{eqnarray}
	\label{LLSY:Eqn:DVF-optimality-system-scaled1}
	\frac{1}{\sqrt{\varrho}}\skpr{\tilde p_{\varrho h},q_h}_{L_2(\Omega)}+ \skpr{\nabla y_{\varrho h},\nabla q_h}_{L_2(\Omega)} 
     &=& 0, \;\; \forall q_h\in V_h, \\
     \label{LLSY:Eqn:DVF-optimality-system-scaled2}
	 -\skpr{\nabla \tilde p_{\varrho h},\nabla v_h}_{L_2(\Omega)} + \frac{1}{\sqrt{\varrho}}\skpr{y_{\varrho h},v_h}_{L_2(\Omega)} 
	 &=& \frac{1}{\sqrt{\varrho}}\skpr{y_d,v_h}_{L_2(\Omega)}, \;\;\forall v_h\in V_h.
\end{eqnarray}
In \cite{LLSY:LangerLoescherSteinbachYang:2023}, we were able to show
the following result for the $L_2$ error between the desired state $y_d$
and the computed finite element state $y_{\varrho h}$.
\begin{theorem}[{\cite[Corollary~1]{LLSY:LangerLoescherSteinbachYang:2023}}]
  \label{thm:error-yy_d}
  Let $(y_{\varrho h},\tilde p_{\varrho h})\in V_h\times V_h$ be the unique
  solution of the coupled finite element variational formulation
  \eqref{LLSY:Eqn:DVF-optimality-system-scaled1}-\eqref{LLSY:Eqn:DVF-optimality-system-scaled2}. 
  Let $y_d\in H_0^s(\Omega)$ for $s\in[0,1]$ or
  $y_d\in H^s(\Omega)\cap H_0^1(\Omega)$ for $s\in (1,2]$. Then 
  \begin{align*}
    \norm{y_{\varrho h}-y_d}_{L_2(\Omega)} \leq c \, h^s \, \norm{y_d}_{H^s(\Omega)},
  \end{align*}
  provided that $\varrho = h^4$. 
\end{theorem} 

\noindent
We recall that $v_h\in V_h$ can be represented in the form 
\begin{align}\label{eq:representation-in-Vh}
  v_h(x) = \sum_{i=1}^{n_h} v_i\varphi_i^h(x),
\end{align}
where $v_i = v_h(x_i)$. Thus, we can associate each finite element
function with its coefficient vector via the finite element isomorphism
$v_h \leftrightarrow \mathbf{v}$, where $\mathbf{v}[i] = v_i. $
With this, the matrix system corresponding to the fe scheme
\eqref{LLSY:Eqn:DVF-optimality-system-scaled1}-\eqref{LLSY:Eqn:DVF-optimality-system-scaled2} 
can be written in the form: find 
$(\mathbf{y}_h,\mathbf{\tilde p}_h) \in \mathbb{R}^{n_h} \times \mathbb{R}^{n_h}$
such that 
\begin{align}\label{eq:matrix-system}
  \begin{pmatrix}
    \frac{1}{\sqrt{\varrho}}M_h & K_h\\
    - K_h^\top & \frac{1}{\sqrt{\varrho}}M_h
  \end{pmatrix}\begin{pmatrix}
    \mathbf{\tilde p}_h\\ \mathbf{y}_h
  \end{pmatrix} =
  \begin{pmatrix}
    \mathbf{0} \\ \frac{1}{\sqrt{\varrho}}\mathbf{y}_{dh}
  \end{pmatrix},
\end{align}
where $M_h$ and $K_h$ denote the mass resp. stiffness matrix with the entries
\begin{align*}
  M_h[i,j] = \int_\Omega \varphi_j^h(x)\varphi_i^h(x)\, dx\quad
  \text{and}\quad K_h[i,j]=\int_\Omega\nabla \varphi_j^h(x)\cdot
  \nabla\varphi_i^h(x) \, dx,
\end{align*}
and the load vector
\begin{align*}
  \mathbf{y}_{dh}[i] = \int_\Omega y_d(x)\varphi_i^h(x)\, dx.
\end{align*}
We note that the system \eqref{eq:matrix-system} is equivalent to
\eqref{LLSY:Eqn:SchurComplementSystemFullM}. Moreover, when eliminating
$\mathbf{\tilde p}_h$ resp. $\mathbf{p}_h$, we arrive at the same 
Schur-complement system \eqref{LLSY:Eqn:SchurComplementSystemFullM}.
    
As already mentioned in the introductionary Section \ref{LLSY:sec:Introduction},
we would like to replace the inverse of the mass matrix $M_h$ in the spd
Schur complement $\varrho K_h M_h^{-1}K_h + M_h$ by the inverse of the lumped
mass matrix $\text{lump}(M_h)$ that is diagonal. 
	%
    %
    The entries of the lumped mass matrix $\text{lump}(M_h)$ are given as
  \begin{align}\label{eq:lumped-mass-matrix}
  	\text{lump}(M_h)[i,j] = \delta_{i,j}\sum_{k=1}^{\overline{n}_h} \widetilde M_h[i,k],\quad i,j=1,\ldots,n_h,
  \end{align}
	where $\widetilde M_h\in\mathbb{R}^{\overline{n}_h}\times \mathbb{R}^{\overline{n}_h}$ denotes the mass matrix on $S_h^1(\mathcal{T}_h)$ with entries 
	$$ \widetilde M_h[i,j]:=\int_\Omega \varphi_j^h(x)\varphi_i^h(x)\, dx,\quad i,j=1,\ldots,\overline{n}_h. $$
	%
	The Schur complement system is then given by \eqref{LLSY:Eqn:SchurComplementSystemLumpedMass} 
	with the Schur complement $S_h = \varrho K_h (\text{lump}(M_h))^{-1} K_h + M_h$ as system matrix.
	Unique solvability of the discrete system follows immediately, since $M_h=M_h^\top >0$ is symmetric and positive definite (spd) and 
	$K_h (\text{lump}(M_h))^{-1}K_h> 0$
	is spd, as $\text{lump}(M_h)$ and $K_h$ are spd. 
	
	Now the aim is to show an equivalent result to Theorem \ref{thm:error-yy_d}, when using the lumped mass matrix. 
	We will exploit 
	ideas from \cite{LLSY:BeckerHansbo:2008a}. 
	The discrete variational formulation for the lumped case reads as follows: find $(\hat y_{\varrho h},\hat p_{\varrho h})\in V_h\times V_h$ such that
	%
	%
	\begin{eqnarray}
	\label{LLSY:Eqn:DVF-optimality-system-scaled-lumped1}
	 \frac{1}{\sqrt{\varrho}}\skpr{\hat p_{\varrho h},q_h}_{h} + \skpr{\nabla \hat y_{\varrho h},\nabla q_h}_{L_2(\Omega)} 
	 &=& 0, \;\;\forall q_h\in V_h,\\
	 \label{LLSY:Eqn:DVF-optimality-system-scaled-lumped2}
	-\skpr{\nabla \hat p_{\varrho h},\nabla v_h}_{L_2(\Omega)} + \frac{1}{\sqrt{\varrho}}\skpr{\hat y_{\varrho h},v_h}_{L_2(\Omega)} 
	  &=& \frac{1}{\sqrt{\varrho}}\skpr{y_d,v_h}_{L_2(\Omega)}, \;\; \forall v_h\in V_h,
	\end{eqnarray}
	where 
	$\skpr{p_h,q_h}_h = \mathbf{q}_h^\top \text{lump}(M_h)\mathbf{p}_h$ denotes the underintegrated inner product on $L_2(\Omega)$
	that is nothing but the realization of the mass lumping.
\begin{lemma}\label{lem:lumped-mass-and-interpolation}
		For $p_h,q_h\in V_h$ the realization of the lumped mass matrix admits the representation 
		\begin{align*}
			\skpr{p_h,q_h}_h = \int_\Omega I_h^1(p_hq_h)\, dx,
		\end{align*} 
	where $I_h^1:\mathcal{C}(\overline \Omega)\to V_h$ denotes the interpolation operator, given as
	\begin{align*}
		I_h^1v(x) = \sum_{i=1}^{n_i} v_i\varphi_i^h(x), \, x\in \overline{\Omega},
	\end{align*}
	where $v_i=v(x_i)$, $v \in \mathcal{C}(\overline{\Omega})$.
    Furthermore, it holds that 
      \begin{align}\label{eq:bound L2 norm discrete inner product}
	\frac{1}{d+2} \, \norm{p_h}_h^2 \leq
 	\norm{p_h}_{L_2(\Omega)}^2\leq \norm{p_h}_h^2 := \skpr{p_h,p_h}_h \text{ for all } p_h\in V_h. 
    \end{align}
\end{lemma}
\begin{proof}
  With the representation \eqref{eq:representation-in-Vh}, we have the
  coefficient vectors $\mathbf{p}_h\leftrightarrow p_h$ and
  $\mathbf{q}_h\leftrightarrow q_h$.
  Then we compute with \eqref{eq:lumped-mass-matrix}, using
  \eqref{eq:properties-phijh},
  \begin{align*}
    \skpr{p_h,q_h}_h
    & = \mathbf{q}_h^\top \text{lump}(M_h)\mathbf{p}_h =
      \sum_{i=1}^{n_h} \sum_{j=1}^{n_h} p_i q_j\delta_{i,j}
      \sum_{k=1}^{\overline{n}_h}\widetilde M_h[i,k] \\
    & = \sum_{i=1}^{n_h} p_iq_i\sum_{k=1}^{\overline{n}_h}
      \widetilde M_h[i,k] = \sum_{i=1}^{n_h} p_i q_i
      \sum_{k=1}^{\overline{n}_h}
      \int_\Omega \varphi_i^h(x)\varphi_k^h(x)\, dx \\
    & = \int_\Omega \underbrace{\sum_{i=1}^{n_h} p_i q_i\varphi_i^h(x)
      }_{I_h^1(p_hq_h)} \underbrace{\sum_{k=1}^{\overline{n}_h}\varphi_k^h(x)
      }_{=1} \, dx. 
  \end{align*}
  The estimate \eqref{eq:bound L2 norm discrete inner product}
    follows from, e.g., \cite[Lemma 9.4]{LLSY:Steinbach:2008Monograph},
    \[
      \frac{|\tau_e|}{(d+1)(d+2)}
      \sum\limits_{x_i \in \overline{\tau}_e} p_i^2 \leq
      \norm{p_h}^2_{L_2(\tau_e)} \leq \frac{|\tau_e|}{d+1}
      \sum\limits_{x_i \in \overline{\tau}_e} p_i^2 ,
    \]
    and
    \[
      \int_{\tau_e} I_h^1(p_h^2)(x) \, dx =
      \sum\limits_{x_i \in \overline{\tau}_e} p_i^2
      \int_{\tau_e} \varphi_i^h(x) \, dx = \frac{|\tau_e|}{d+1}
      \sum\limits_{x_i \in \overline{\tau}_e} p_i^2,
    \]
    when summing up over all elements $\tau_e$.

\end{proof}

\noindent
With this representation, we can compute the consistency error. 
\begin{lemma}\label{lem:interpolation-error}
  Let $h = \max_{e=1,\ldots,N_h} h_e$. Then, for $p_h,q_h\in V_h$, it holds
  \begin{align*}
    \left|\skpr{p_h,q_h}_{L_2(\Omega)}-\skpr{p_h,q_h}_{h}\right|
    &=\left|\int_\Omega \Big[ p_h(x)q_h(x)-I_h^1(p_hq_h)(x) \Big] \,
      dx\right|\\
    & \leq c \, h^2 \Big( \varepsilon^2 \,
      \norm{\nabla p_h}_{L_2(\Omega)}^2 +
      \frac{1}{\varepsilon^2}\norm{\nabla q_h}_{L_2(\Omega)}^2\Big),
  \end{align*}
  for any $\varepsilon>0$. 
\end{lemma}
%
\begin{proof}
  The first representation follows from
  Lemma \ref{lem:lumped-mass-and-interpolation}. Let $\tau_e$ be a
  simplicial finite element with the nodes $x_{e_i}$, $i=1,\ldots,d+1$.
  The associated nodal values of a piecewise linear finite element
  function $p_h$ are the coefficients $p_{e_i}$, $i=1,\ldots,d+1$.
  In particular, for $d=1$ and $x \in \tau_e$, we then compute
  \begin{eqnarray*}
    && \hspace*{-1cm}
       \int_{\tau_e} \Big[ p_h(x) q_h(x) - I_h^1(p_hq_h)(x) \Big] \, dx \\
  && = \, \int_{x_{e_1}}^{x_{e_2}} \left(
     \Big[ p_{e_1} + \frac{x-x_{e_1}}{h_e} (p_{e_2}-p_{e_1}) \Big]
     \Big[ q_{e_1} + \frac{x-x_{e_1}}{h_e} (q_{e_2}-q_{e_1}) \Big] \right. \\
    && \hspace*{4cm} \left.
     - \Big[ p_{e_1}q_{e_1} + \frac{x-x_{e_1}}{h_e} (p_{e_2}q_{e_2}-p_{e_1}q_{e_1})
     \Big] \right) dx \\
    && = \frac{1}{6} \, h_e \, (p_{e_2}-p_{e_1}) \, (q_{e_2}-q_{e_1}) \, = \,
  \frac{1}{6} \, h_e^2 \, \int_{x_{e_1}}^{x_{e_2}}
     \frac{p_{e_2}-p_{e_1}}{h_e} \, \frac{q_{e_2}-q_{e_1}}{h_e} \, dx \\
  && = \frac{1}{6} \, h_e^2 \, \int_{x_{e_1}}^{x_{e_2}} p_h'(x) \, q_h'(x) \, dx
    \, \leq \, \frac{1}{6} \, h_e^2 \, \| \nabla_x p_h \|_{L^2(\tau_e)}
     \| \nabla_x q_h \|_{L^2(\tau_e)} \, .
  \end{eqnarray*}
  For $d=2$ and $x \in \tau_e$, we introduce the representation
  $x=x_{e_1} +J_e \eta$ with respect to the reference element
  $\tau = \{ \eta \in {\mathbb{R}}^2 : \eta_1 \in (0,1), \eta_2 \in
  (0,1-\eta_1)\}$ and we write $p_h(x)=p_h(x_{e_1}+J_e\eta)=
  \widetilde{p}_h(\eta)$, $\eta \in \tau$.
  Similar as in the case $d=1$ we then compute, using
  $\mbox{det} \, J_e = 2 \, |\tau_e|$,
  \begin{eqnarray*}
  && \int_{\tau_e} \Big[ p_h(x) q_h(x) - I_h^1(p_hq_h)(x) \Big] \, dx \, = \,
     \int_\tau \Big[ \widetilde{p}_h(\eta) \widetilde{q}_h(\eta) -
     I_h^1(\widetilde{p}_h\widetilde{q}_h)(\eta) \Big] \,
     \mbox{det} J_e \, d\eta \\
  && = \frac{|\tau_e|}{12} \Big[
     (p_0-p_2)(q_2-q_0) + (p_1-p_0)(q_0-q_1) + (p_1-p_2)(q_2-q_1) 
     \Big] \\
  && \leq \frac{|\tau_e|}{12}
     \Big[
     (p_0 - p_2)^2 + (p_1-p_0)^2 + (p_1-p_2)^2
     \Big]^{1/2} \\
    && \hspace*{3cm} \cdot
     \Big[
     (q_2 - q_0)^2 + (q_0-q_1)^2 + (q_2-q_1)^2
     \Big]^{1/2} \, .
\end{eqnarray*}
With
\[
  (p_1-p_2)^2 =
  (p_1-p_0+p_0-p_2)^2 \leq 2 \, (p_1-p_0)^2 + 2 \, (p_0-p_2)^2, 
\]
we further have, e.g., \cite[Lemma 9.1]{LLSY:Steinbach:2008Monograph},
\begin{eqnarray*}
  && \int_{\tau_e} \Big[ p_h(x) q_h(x) - I_h^1(p_hq_h)(x) \Big] \, dx \\
  && \hspace*{1cm}  \leq \, \frac{|\tau_e|}{4} 
     \Big[ (p_0 - p_2)^2 + (p_1-p_0)^2 \Big]^{1/2}
     \Big[ (q_2 - q_0)^2 + (q_0-q_1)^2 \Big]^{1/2} \\
  & & \hspace*{1cm} = \, \frac{|\tau_e|}{4} \left[
        2 \int_\tau |\nabla_\eta \widetilde{p}_h|^2 d\eta \right]^{1/2}
        \left[
        2 \int_\tau |\nabla_\eta \widetilde{q}_h|^2 d\eta \right]^{1/2} \\
  && \hspace*{1cm}
     = \, \frac{|\tau_e|}{2} \, \| \nabla_\eta \widetilde{p}_h \|_{L^2(\tau)}
     \| \nabla_\eta\widetilde{q}_h \|_{L^2(\tau)} 
     \, \leq \, c \, h_e^2 \, \| \nabla_x p_h \|_{L^2(\tau_e)}
     \| \nabla_x q_h\|_{L^2(\tau_e)} \, .
\end{eqnarray*}
For $d=3$, we proceed in the same way. Now the reference element is
given by $\tau = \{ \eta \in {\mathbb{R}}^3 : \eta_1 \in (0,1),
\eta_2 \in (0,1-\eta_1), \eta_3 \in (0,1-\eta_1-\eta_2) \}$,
and $\mbox{det} \, J_e = 6 \, |\tau_e|$. Then,
\begin{eqnarray*}
  && \int_{\tau_e} \Big[ p_h(x) q_h(x) - I_h^1(p_hq_h)(x) \Big] \, dx \\
  && = \,
     6 \, |\tau_e| \, \int_0^1 \int_0^{1-\eta_1} \int_0^{1-\eta_1-\eta_2}
     \Big[ \widetilde{p}_h(\eta) \widetilde{q}_h(\eta) -
     I_h^1(\widetilde{p}_h\widetilde{q}_h)(\eta) \Big] d\eta_3 d\eta_2 d\eta_1
  \\
  && = \frac{|\tau_e|}{20} \Big[
     (p_0-p_1)(q_1-q_0) + (p_0-p_2)(q_2-q_0) + (p_0-p_3)(q_3-q_0) \\
  && \hspace*{3cm} +
     (p_1-p_2)(q_2-q_1) + (p_1-p_3)(q_3-q_1) + (p_2-p_3)(q_3-q_2) \Big] \\
  && \leq \frac{|\tau_e|}{20} \Big[
     (p_0-p_1)^2 + (p_0-p_2)^2 + (p_0-p_3)^2 + (p_1-p_2)^2 + (p_1-p_3)^2
     + (p_2-p_3)^2 \Big]^{1/2} \\
  && \hspace*{1cm} \cdot \Big[
     (q_1-q_0)^2 + (q_2-q_0)^2 + (q_3-q_0)^2 + (q_2-q_1)^2 + (q_3-q_1)^2
     + (q_3-q_2)^2 \Big]^{1/2} \\
  && \leq \frac{|\tau_e|}{4} \Big[
     (p_0-p_1)^2 + (p_0-p_2)^2 + (p_0-p_3)^2 \Big]^{1/2} \Big[
     (q_1-q_0)^2 + (q_2-q_0)^2 + (q_3-q_0)^2 \Big]^{1/2} \\
  && = \frac{|\tau_e|}{4} \left[ 6 \int_\tau |\nabla_\eta \widetilde{p}_h|^2
     d\eta \right]^{1/2} \left[ 6 \int_\tau |\nabla_\eta| \widetilde{q}_h|^2
     d\eta \right]^{1/2} \\
  && =  \frac{3}{2} \, |\tau_e| \,
        \| \nabla_\eta \widetilde{p}_h \|_{L^2(\tau)}
        \| \nabla_\eta \widetilde{q}_h \|_{L^2(\tau)} \,
  \leq \, c \, h_e^2 \, \| \nabla_x p_h \|_{L^2(\tau_e)}
     \| \nabla_x q_h \|_{L^2(\tau_e)} \, .
\end{eqnarray*}
Hence, using Young's inequality,
\[
  \| \nabla_x p_h \|_{L_2(\tau_e)} \| \nabla_x q_h \|_{L_2(\tau_e)}
  \leq \frac{1}{2} \, \left( \varepsilon^2 \,
    \| \nabla_x p_h \|^2_{L_2(\tau_e)} +
    \frac{1}{\varepsilon^2} \, \| \nabla_x q_h \|^2_{L_2(\tau_e)} \right), 
\]
and summing up over all elements $\tau_e$, this gives the desired
estimate.
\end{proof}

\noindent
We need one more preliminary result, before we can state the main theorem.

\begin{lemma}\label{lem:regularization-estimates}
  Let $(y_\varrho,\tilde p_\varrho)\in H_0^1(\Omega)\times H_0^1(\Omega)$ be
  the unique solution of the reduced optimality system 
  \eqref{LLSY:Eqn:VF-optimality-system-scaled1} and
  \eqref{LLSY:Eqn:VF-optimality-system-scaled2}.
  Then there holds the regularization error estimate
  \begin{align*}
    \norm{y_\varrho-y_d}_{H^{-1}(\Omega)}
    \leq c \, \sqrt{\varrho} \, |y_d|_{H^1(\Omega)} \quad
    \text{for } y_d \in H_0^1(\Omega). 
  \end{align*}
  Additionally, for $y_d \in H^{\Delta}(\Omega)\cap H_0^1(\Omega)$, there holds 
  \begin{align*}
    \norm{y_\varrho-y_d}_{L_2(\Omega)} \leq
    \sqrt{\varrho} \, \| \Delta y_d\|_{L_2(\Omega)} \quad \text{and} \quad
    \norm{\Delta y_\varrho}_{L_2(\Omega)} \leq \norm{\Delta y_d}_{L_2(\Omega)}. 
  \end{align*}
\end{lemma}

\begin{proof}
  The first estimate is given in
  \cite[Theorem 4.1, (4.7)]{LLSY:NeumuellerSteinbach:2021a}. 
  The second and third estimate can be found in
  \cite[Lemma 1, (2.5)]{LLSY:LangerLoescherSteinbachYang:2023} and
  in the proof of this lemma. But for clarity, we will recall the proof. 
  We note that, by the optimality system, we have the equations
  $-\Delta y_\varrho = u_\varrho$, and $p_\varrho = -\varrho u_\varrho$.
  First, assuming the regularity
  $y_\varrho,y_d\in H^\Delta(\Omega)\cap H_0^1(\Omega)$, using 
  \eqref{LLSY:Eqn:VF-optimality-system1} and
  \eqref{LLSY:Eqn:VF-optimality-system2},
  and integration by parts, we obtain
  \begin{align*}
    \norm{y_\varrho -y_d}_{L_2(\Omega)}^2
    & = \skpr{y_\varrho -y_d,y_\varrho -y_d}_{L_2(\Omega)}
      = \skpr{\nabla p_\varrho,\nabla(y_\varrho-y_d)}_{L_2(\Omega)} \\
    & = \skpr{p_\varrho,-\Delta y_\varrho}_{L_2(\Omega)} +
      \skpr{p_\varrho,\Delta y_d}_{L_2(\Omega)} \\
    & = -\varrho \, \skpr{u_\varrho,-\Delta y_\varrho}_{L_2(\Omega)} -
      \varrho \, \skpr{u_\varrho,\Delta y_d}_{L_2(\Omega)} \\
    & = - \varrho \, \norm{\Delta y_\varrho}_{L_2(\Omega)}^2 +
      \varrho \, \skpr{\Delta y_\varrho,\Delta y_d}_{L_2(\Omega)}.
  \end{align*}
  From this we conclude
  \begin{align*}
    \norm{y_\varrho-y_d}_{L_2(\Omega)}^2 +
    \varrho \, \norm{\Delta y_\varrho}_{L_2(\Omega)}^2 \leq
    \varrho \, \norm{\Delta y_\varrho}_{L_2(\Omega)}
    \norm{\Delta y_d}_{L_2(\Omega)},
  \end{align*}
  and further
  \begin{align*}
    \norm{\Delta y_\varrho}_{L_2(\Omega)} \leq
    \norm{\Delta y_d}_{L_2(\Omega)} \quad \text{and} \quad
    \norm{y_\varrho-y_d}_{L_2(\Omega)} \leq
    \sqrt{\varrho} \, \norm{\Delta y_d}_{L_2(\Omega)}. 
  \end{align*}
  From the first estimate we conclude that
  $y_d\in H^\Delta(\Omega)\cap H_0^1(\Omega)$
  implies 
  $y_\varrho \in H^\Delta(\Omega)\cap H_0^1(\Omega)$. 
  This confirms that it is sufficient to require the  regularity on $y_d$ only. 
\end{proof}

\noindent
The main statement of this paper is formulated in the following theorem. 

\begin{theorem}\label{thm:h1h2error-lumped}
  Let $(\hat y_{\varrho h},\hat p_{\varrho h}) \in V_h\times V_h$ be the
  unique solution of the variational formulation 
  \eqref{LLSY:Eqn:DVF-optimality-system-scaled-lumped1} and
  \eqref{LLSY:Eqn:DVF-optimality-system-scaled-lumped2}.
  Assume that $\mathcal{T}_h$ is globally quasi-uniform such
  that a global inverse inequality holds true. Further,
  choose $\varrho = h^4$. Then, 
  \begin{align*}
    \norm{\hat y_{\varrho h}-y_d}_{L_2(\Omega)} \leq
    \begin{cases}
      c \, h \, \norm{y_d}_{H^1_0(\Omega)},
      & \text{if } y_d\in H_0^1(\Omega),\\[1mm]
      c \, h^2 \, \| y_d \|_{H^2(\Omega)},
      & \text{if } y_d \in H^2(\Omega) \cap H_0^1(\Omega) \text{ and }
      \Omega \text{ is convex.}
    \end{cases}
  \end{align*}
\end{theorem}

\begin{proof}
  Let $(y_{\varrho h},\tilde p_{\varrho h}) \in V_h\times V_h$ be the unique
  solution of \eqref{LLSY:Eqn:DVF-optimality-system-scaled1} and
  \eqref{LLSY:Eqn:DVF-optimality-system-scaled2}. By the triangle
  inequality, we get that
  \begin{align*}
    \norm{\hat y_{\varrho h}-y_d}_{L_2(\Omega)} \leq
    \norm{\hat y_{\varrho h}-y_{\varrho h}}_{L_2(\Omega)} +
    \norm{y_{\varrho h}-y_d}_{L_2(\Omega)}.
  \end{align*}
  By Theorem~\ref{thm:error-yy_d}, the second term fulfils the estimate.
  Thus, it is sufficient to bound the first term. Therefore, 
  subtracting the variational formulation
  \eqref{LLSY:Eqn:DVF-optimality-system-scaled1} and
  \eqref{LLSY:Eqn:DVF-optimality-system-scaled2} from  
  \eqref{LLSY:Eqn:DVF-optimality-system-scaled-lumped1} and
  \eqref{LLSY:Eqn:DVF-optimality-system-scaled-lumped2}
  with $v_h=\hat y_{\varrho h}-y_{\varrho h}$ and
  $q_h=\hat p_{\varrho h}-\tilde p_{\varrho h}$,
  we obtain the equalities
  \begin{align*}
    \frac{1}{\sqrt{\varrho}} \, \norm{\hat y_{\varrho h}-
    & y_{\varrho h}}_{L_2(\Omega)}^2 \, = \,
      \frac{1}{\sqrt{\varrho}} \, \skpr{\hat y_{\varrho h}-y_{\varrho h},
      \hat y_{\varrho h}-y_{\varrho h}}_{L_2(\Omega)} \\
    & = \skpr{\nabla(\hat p_{\varrho h}-\tilde p_{\varrho h}),
      \nabla (\hat y_{\varrho h}-y_{\varrho h})}_{L_2(\Omega)} \\
    & = \frac{1}{\sqrt{\varrho}} \, \left( \skpr{\tilde p_{\varrho h},
      \hat p_{\varrho h}-\tilde p_{\varrho h}}_{L_2(\Omega)} -
      \skpr{\hat p_{\varrho h},\hat p_{\varrho h}-\tilde p_{\varrho h}}_h
      \right) \\
    & = \frac{1}{\sqrt{\varrho}} \, \left( \skpr{\tilde p_{\varrho h},
      \hat p_{\varrho h}-\tilde p_{\varrho h}}_{L_2(\Omega)} -
      \norm{\hat p_{\varrho h}-\tilde p_{\varrho h}}_h^2 -
      \skpr{\tilde p_{\varrho h},\hat p_{\varrho h}-\tilde p_{\varrho h}}_h
      \right). 
  \end{align*}
  Multiplying by $\sqrt{\varrho}$ and using
  Lemma \ref{lem:lumped-mass-and-interpolation}, we further get 
  \begin{align*}
    \norm{\hat y_{\varrho h}-y_{\varrho h}}_{L_2(\Omega)}^2
    & + \norm{\hat p_{\varrho h}-\tilde p_{\varrho h}}_{L_2(\Omega)}^2
      \leq \norm{\hat y_{\varrho h}-y_{\varrho h}}_{L_2(\Omega)}^2 +
      \norm{\hat p_{\varrho h}-\tilde p_{\varrho h}}_{h}^2 \\[1mm]
    & = \skpr{\tilde p_{\varrho h},
      \hat p_{\varrho h}-\tilde p_{\varrho h}}_{L_2(\Omega)} -
      \skpr{\tilde p_{\varrho h},\hat p_{\varrho h}-\tilde p_{\varrho h}}_h \\
    & = \int_\Omega \left[ \tilde p_{\varrho h}(x)
      (\hat p_{\varrho h}(x)-\tilde p_{\varrho h}(x)) -
      I_h^1(\tilde p_{\varrho h}(\hat p_{\varrho h}-\tilde p_{\varrho h})(x)
      \right] \, dx . 
  \end{align*}
  With Lemma \ref{lem:interpolation-error}, choosing
  $p_h = \tilde p_{\varrho h}$ and $q_h=\hat p_{\varrho h}-\tilde p_{\varrho h}$,
  we estimate, 
  for some $\varepsilon > 0$ to be specified,
  \begin{align*}
    \norm{\hat y_{\varrho h}-y_{\varrho h}}_{L_2(\Omega)}^2
    & + \norm{\hat p_{\varrho h} - \tilde p_{\varrho h}}_{L_2(\Omega)}^2 \\
    & \leq c \, h^2 \, \left( \varepsilon^2 \,
      \norm{\nabla \tilde p_{\varrho h}}_{L_2(\Omega)}^2 +
      \frac{1}{\varepsilon^2} \,
      \norm{\nabla(\hat p_{\varrho h}-\tilde p_{\varrho h})}_{L_2(\Omega)}^2
      \right).
  \end{align*}
  Using an inverse inequality, we estimate the second term by 
  \begin{align*}
    \frac{ch^2}{\varepsilon^2} \,
    \norm{\nabla(\hat p_{\varrho h}-\tilde p_{\varrho h})}_{L_2(\Omega)}^2
    \leq \frac{cc_I}{\varepsilon^2} \,
    \norm{\hat p_{\varrho h}-\tilde p_{\varrho h}}_{L_2(\Omega)}^2 =
    \norm{\hat p_{\varrho h}-\tilde p_{\varrho h}}_{L_2(\Omega)}^2,
  \end{align*}
  when choosing $\varepsilon = \sqrt{cc_I}$. Thus, it holds 
  \begin{align}\label{eq:thm1-ineq1}
    \norm{\hat y_{\varrho h}-y_{\varrho h}}_{L_2(\Omega)}^2 \leq
    \tilde c \, h^2 \, \norm{\nabla \tilde p_{\varrho h}}_{L_2(\Omega)}^2.
  \end{align}
  Now it is sufficient to bound
  $\norm{\nabla \tilde p_{\varrho h}}_{L_2(\Omega)}$ suitably. Let
  $(y_\varrho,\tilde p_\varrho)\in H_0^1(\Omega)\times H_0^1(\Omega)$ be
  the unique solution of the coupled variational formulation
  \eqref{LLSY:Eqn:VF-optimality-system1} and
  \eqref{LLSY:Eqn:VF-optimality-system2}.
  Using the triangle inequality and the trivial inequality
  $(a+b)^2\leq 2(a^2+b^2)$, we get
  \begin{align}\label{eq:thm1-ineq5}
    \norm{\nabla \tilde p_{\varrho h}}_{L_2(\Omega)}^2 \leq
    2 \, \Big( \norm{\nabla \tilde p_\varrho}_{L_2(\Omega)}^2 +
    \norm{\nabla (\tilde p_{\varrho h}-\tilde p_{\varrho})}_{L_2(\Omega)}^2
    \Big). 
  \end{align}
  For $(y_\varrho, \widetilde p_{\varrho}) \in
    H_0^1(\Omega)\times H_0^1(\Omega)$ and
    $(y_{\varrho h},\widetilde p_{\varrho h})\in V_h\times V_h$ as solutions
    of
    \eqref{LLSY:Eqn:VF-optimality-system1}-\eqref{LLSY:Eqn:VF-optimality-system2}
    and
    \eqref{LLSY:Eqn:DVF-optimality-system-scaled1}-\eqref{LLSY:Eqn:DVF-optimality-system-scaled2},
    respectively, we can show, as in the proof of
    \cite[Theorem 1]{LLSY:LangerLoescherSteinbachYang:2023}, using an
    inverse inequality and $\varrho =h^4$, that Cea's Lemma 
    \begin{eqnarray*}
      && h^{-2} \, \norm{y_\varrho-y_{\varrho h}}_{L_2(\Omega)}^2 +
         \norm{\nabla (y_\varrho - y_{\varrho h})}_{L_2(\Omega)}^2 \\
      && \hspace{10mm} + \, h^{-2} \,
         \norm{\tilde p_\varrho-\tilde p_{\varrho h}}_{L_2(\Omega)}^2 +
         \norm{\nabla (\tilde p_\varrho -\tilde p_{\varrho h})}_{L_2(\Omega)}^2\\
      &&\hspace{20mm}  \leq \, c \, \Big[ h^{-2} \,
         \norm{y_\varrho -v_h}_{L_2(\Omega)}^2 +
         \norm{\nabla (y_\varrho -v_h)}_{L_2(\Omega)}^2 \\
      && \hspace{30mm} + \, h^{-2} \,
         \norm{\tilde p_\varrho-q_h}_{L_2(\Omega)}^2 +
         \norm{\nabla(\tilde p_\varrho-q_h)}_{L_2(\Omega)}^2 \Big]
    \end{eqnarray*}
    holds true for all $(v_h,q_h)\in V_h\times V_h$. Further, using best
    approximation results, we get, for $y_d\in H_0^1(\Omega)$, that 
    \begin{eqnarray*}
      && h^{-2} \, \norm{y_\varrho-y_{\varrho h}}_{L_2(\Omega)}^2 +
         \norm{\nabla (y_\varrho - y_{\varrho h})}_{L_2(\Omega)}^2\\
      && \hspace{10mm} + \, h^{-2} \,
         \norm{\tilde p_\varrho-\tilde p_{\varrho h}}_{L_2(\Omega)}^2 +
         \norm{\nabla (\tilde p_\varrho -\tilde p_{\varrho h})}_{L_2(\Omega)}^2
         \, \leq \, c \, |y_d|_{H^1(\Omega)}^2,
    \end{eqnarray*}
    and, for $y_d\in H_0^1(\Omega)\cap H^2(\Omega)$, that
    \begin{eqnarray*}
      && h^{-2} \, \norm{y_\varrho-y_{\varrho h}}_{L_2(\Omega)}^2 +
         \norm{\nabla (y_\varrho - y_{\varrho h})}_{L_2(\Omega)}^2\\
      &&\hspace{10mm} + \, h^{-2} \,
         \norm{\tilde p_\varrho-\tilde p_{\varrho h}}_{L_2(\Omega)}^2 +
         \norm{\nabla (\tilde p_\varrho -\tilde p_{\varrho h})}_{L_2(\Omega)}^2
         \, \leq \, c \, h^2 \, |y_d|_{H^2(\Omega)}^2.
    \end{eqnarray*}
    From this we immediately conclude that 
    \begin{align}\label{eq:thm1-ineq2}
      \norm{\nabla (\tilde p_{\varrho h}-\tilde p_\varrho)}_{L_2(\Omega)}^2
      \leq
      \begin{cases}
        c \, |y_d|_{H^1(\Omega)}^2,
        &\text{for } y_d\in H^1_0(\Omega),\\[1mm]
	c \, h^2 \, |y_d|_{H^2(\Omega)}^2,
        &\text{for } y_d\in  H_0^1(\Omega)\cap H^2(\Omega). 	
      \end{cases}
    \end{align}
  For the first term, we use the first inequality of
  Lemma \ref{lem:regularization-estimates} for $y_d\in H_0^1(\Omega)$
  to estimate
  \begin{align*}
    \norm{\nabla \tilde p_{\varrho}}_{L_2(\Omega)}^2
    & = \skpr{\nabla \tilde p_\varrho,\nabla \tilde p_\varrho}_{L_2(\Omega)}
      = \frac{1}{\sqrt{\varrho}} \,
      \skpr{y_\varrho-y_d,\tilde p_\varrho}_{L_2(\Omega)} \\
    &\leq \frac{1}{\sqrt{\varrho}} \, \norm{y_\varrho-y_d}_{H^{-1}(\Omega)}
      \norm{\nabla \tilde p_{\varrho}}_{L_2(\Omega)} \, \leq \,
      c \, |y_d|_{H^1(\Omega)} \norm{\nabla \tilde p_\varrho}_{L_2(\Omega)}, 
  \end{align*}
  from which we conclude 
  \begin{align}\label{eq:thm1-ineq3}
    \norm{\nabla \tilde p_\varrho}_{L_2(\Omega)} \leq
    c \, |y_d|_{H^1(\Omega)} . 
  \end{align}
  For a convex domain $\Omega$, we have $H^{\Delta}(\Omega)=H^2(\Omega)$.
  Thus $y_d\in H^2(\Omega)\cap H_0^1(\Omega)$, and we get with the same
  reasoning, using the second inequality of
  Lemma \ref{lem:regularization-estimates},
  \begin{align*}
    \norm{\nabla \tilde p_\varrho}_{L_2(\Omega)}^2
    & = \frac{1}{\sqrt{\varrho}} \,
      \skpr{y_\varrho -y_d,\tilde p_\varrho}_{L_2(\Omega)}\\
    &\leq \frac{1}{\sqrt{\varrho}} \,
      \norm{y_\varrho-y_d}_{L_2(\Omega)}\norm{\tilde p_\varrho}_{L_2(\Omega)}
      \, \leq \, |y_d|_{H^2(\Omega)}\norm{\tilde p_\varrho}_{L_2(\Omega)}.
  \end{align*}
  Now, we recall that, from the optimality system
  \eqref{LLSY:Eqn:optimality-system1}-\eqref{LLSY:Eqn:VF-optimality-system2},
  we have
  \begin{align*}
    \tilde p_\varrho = \frac{1}{\sqrt{\varrho}} \, p_\varrho =
    - \sqrt{\varrho} \, u_\varrho =
    \sqrt{\varrho} \, \Delta y_\varrho \quad \text{in } \Omega,
  \end{align*}
  and thus, with Lemma \ref{lem:regularization-estimates}, we obtain
  \begin{align*}
    \norm{\tilde p_\varrho}_{L_2(\Omega)} =
    \sqrt{\varrho} \, \norm{\Delta y_\varrho}_{L_2(\Omega)} \leq
    \sqrt{\varrho} \, \norm{\Delta y_d}_{L_2(\Omega)} \leq
    \sqrt{\varrho} \, |y_d|_{H^2(\Omega)}.
  \end{align*}
  Thus, for $\varrho = h^4$ and $y_d \in H^2(\Omega)\cap H_0^1(\Omega)$,
  we arrive at the estimate
  \begin{align}\label{eq:thm1-ineq4}
    \norm{\nabla \tilde p_\varrho}_{L_2(\Omega)} \leq \,
    h \, |y_d|_{H^2(\Omega)} .
  \end{align}
  Now, combining \eqref{eq:thm1-ineq1} with \eqref{eq:thm1-ineq5},
  \eqref{eq:thm1-ineq2},\eqref{eq:thm1-ineq3}, and
  \eqref{eq:thm1-ineq4}, this gives
  \begin{align*}
    \norm{\hat y_{\varrho h}-y_{\varrho h}}_{L_2(\Omega)} \leq h \,
    \norm{\nabla \tilde p_{\varrho h}}_{L_2(\Omega)} \leq
    \begin{cases}
      c \, h \, |y_d|_{H^1(\Omega)},
      & \text{for }y_d\in H_0^1(\Omega),\\[1mm]
      c \, h^2 \, |y_d|_{H^2(\Omega)},
      & \text{for }y_d\in H^2(\Omega)\cap H_0^1(\Omega). 
    \end{cases}      
  \end{align*}
%
\end{proof}

\begin{lemma}\label{lem:l2-error-lumped}
  Let $(\hat y_{\varrho h},\hat p_{\varrho h})\in V_h\times V_h$ be the
  unique solution of the coupled variational formulation
  \eqref{LLSY:Eqn:DVF-optimality-system-scaled-lumped1} and
  \eqref{LLSY:Eqn:DVF-optimality-system-scaled-lumped2}.
  Then, for $y_d\in L_2(\Omega)$, we get the error estimate
  \begin{align*}
    \norm{\hat y_{\varrho h}-y_d}_{L_2(\Omega)} \leq \norm{y_d}_{L_2(\Omega)}.
  \end{align*}
\end{lemma}

\begin{proof}
  Choosing $q_h=\hat p_{\varrho h}$ and $v_h=\hat y_{\varrho h}$ in 
  \eqref{LLSY:Eqn:DVF-optimality-system-scaled-lumped1} and
  \eqref{LLSY:Eqn:DVF-optimality-system-scaled-lumped2},
  summing up the equations, and multiplying with $\sqrt{\varrho}$, this gives
  \begin{align*}
    \skpr{\hat p_{\varrho h},\hat p_{\varrho}}_h +
    \skpr{\hat y_{\varrho h},\hat y_{\varrho h}}_{L_2(\Omega)} =
    \skpr{y_d,\hat y_{\varrho h}}_{L_2(\Omega)}. 
  \end{align*}
  Rewriting this equality gives
  \begin{align*}
    \skpr{\hat p_{\varrho h},\hat p_{\varrho h}}_h +
    \skpr{\hat y_{\varrho h}-y_d,\hat y_{\varrho h}-y_d}_{L_2(\Omega)} =
    \skpr{y_d-\hat y_{\varrho h},y_d}_{L_2(\Omega)}, 
  \end{align*}
  which yields the desired estimate. 
\end{proof}

\begin{theorem}\label{thm:hserror-lumped}
  Let $(\hat y_{\varrho h},\hat p_{\varrho h})\in V_h\times V_h$ be the
  unique solution of the coupled variational formulation
  \eqref{LLSY:Eqn:DVF-optimality-system-scaled-lumped1}-\eqref{LLSY:Eqn:DVF-optimality-system-scaled-lumped2}.
  For $y_d\in H_0^s(\Omega)$, $s\in [0,1]$, and
  $y_d\in H^s(\Omega)\cap H_0^1(\Omega)$, $s\in (1,2]$, there holds
  the error estimate
  \begin{align}
    \label{LLSY:Eqn:HsDiscretizationErrorEstimate}
    \norm{\hat y_{\varrho h}-y_d}_{L_2(\Omega)} \, \leq \,
    c \, h^s \, \norm{y_d}_{H^s(\Omega)}. 	
  \end{align}
\end{theorem}

\begin{proof}
  This is a direct consequence of Theorem~\ref{thm:h1h2error-lumped}
  and Lemma~\ref{lem:l2-error-lumped},
  together with a space interpolation argument. 
\end{proof}
	

\section{Nested PCG Iteration}
\label{LLSY:sec:NestedPCGIteration}	
Finally, we have to solve the spd mass-lumped Schur-complement system \eqref{LLSY:Eqn:SchurComplementSystemLumpedMass}
that we now write in the compact from: find $\mathbf{y}_h \in \mathbb{R}^{n_h}$ such that
\begin{equation}
 \label{LLSY:Eqn:SchurComplementSystemMassLumpedCompact}
		S_h \mathbf{y}_h = \mathbf{y}_{dh},
\end{equation}
%
where 
$S_h = \varrho K_h D_h^{-1} K_h + M_h$, and 
$D_h$ is the  lumped mass matrix $\text{lump}(M_h)$.
For simplicity, we omit the hat over  $\mathbf{y}_h$ in \eqref{LLSY:Eqn:SchurComplementSystemMassLumpedCompact}
and throughout this section.
The fast solution of the symmetric and indefinite system \eqref{LLSY:Eqn:DiscreteReducedOptimalitySystem} 
with the original mass matrix $M_h$ 
instead of $D_h$ 
was studied in \cite{LLSY:LangerLoescherSteinbachYang:2023}.
Since the matrix $D_h$ is diagonal, the matrix-by-vector multiplication $S_h * \mathbf{y}_h$ 
can now be performed efficiently. 
Therefore, we can use the PCG method for solving \eqref{LLSY:Eqn:SchurComplementSystemMassLumpedCompact}.
Moreover, it turns out that $D_h$ can also serve as preconditioner in the case 
$\varrho = h^4$ that leads to the optimally balanced estimate of $\|{\hat y}_{\varrho h}-y_d\|_{L_2(\Omega)}$
as was shown in Section~\ref{LLSY:sec:MassLumpingAndErroAnalysis}; 
see Theorem~\ref{thm:h1h2error-lumped} and Theorem~\ref{thm:hserror-lumped}.
First of all, we can easily show that $D_h$ is spectrally equivalent to $M_h$,
i.e., there exist positive, $h$-independent constants 
$\underline{c}_\text{\tiny MD}$ and $\overline{c}_\text{\tiny MD}$
such that


\begin{equation}
\label{LLSY:Eqn:SpectralEquivalenceDMD}
 \underline{c}_\text{\tiny MD} D_h \le  M_h \le \overline{c}_\text{\tiny MD} D_h,
\end{equation}
where $\underline{c}_\text{\tiny MD} = \lambda_\text{\tiny min} = \lambda_\text{\tiny min}(D_\tau^{-1}M_\tau) = 1/(d+2)$
and  $\overline{c}_\text{\tiny MD} = \lambda_\text{\tiny max} = \lambda_\text{\tiny max}(D_\tau^{-1}M_\tau) = 1$
are the minimal eigenvalue and maximal eigenvalue of the small generalized eigenvalue problem
$M_\tau \mathbf{v}_\tau = \lambda D_\tau \mathbf{v}_\tau$ in $\mathbb{R}^{d+1}$, respectively.
$M_\tau$ and $D_\tau$ denote resp. the mass matrix and the lumped mass matrix 
corresponding to the reference element (unit simplex) $\tau$ to which every element $\tau_e$ from $\mathcal{T}_h$
is mapped by an affine-linear mapping $x=x_{e_1}+J_e\eta$.
We note that the spectral equivalence inequalities \eqref{LLSY:Eqn:SpectralEquivalenceDMD} 
are nothing but the algebraic version of the inequalities \eqref{eq:bound L2 norm discrete inner product} 
where the constants were already explicitly computed.

In order to estimate the Schur complement $S_h = \varrho K_h D_h^{-1} K_h + M_h$ by the 
mass matrix $M_h$ from below and above in the spectral sense, 
it is obviously enough to estimate $\varrho K_h D_h^{-1} K_h$ from above by $M_h$.
Using the spectral equivalence inequalities \eqref{LLSY:Eqn:SpectralEquivalenceDMD}, 
Cauchy's inequality, local inverse inequalities, and 
$\varrho = h^4$,
we get
\begin{eqnarray}\label{LLSY:Eqn:SpectralEquivalenceSD}
  \nonumber
  (\varrho K_h D_h^{-1} K_h\mathbf{v}_h,\mathbf{v}_h) 
  & = & \varrho \, (D_h^{-1} K_h\mathbf{v}_h,K_h\mathbf{v}_h)
        \, \le \, \overline{c}_\text{\tiny MD} \, \varrho \,
        (M_h^{-1} K_h\mathbf{v}_h,K_h\mathbf{v}_h)\\[1mm] \nonumber
  & = &  \overline{c}_\text{\tiny MD}
        (K_h (\varrho^{-1} M_h)^{-1} K_h\mathbf{v}_h,\mathbf{v}_h) \\[1mm]
  \nonumber
  & = & \overline{c}_\text{\tiny MD} \sup_{\mathbf{q}_h \in \mathbb{R}^{n_h}} 
        \frac{(K_h\mathbf{v}_h,\mathbf{q}_h)^2}
        {(\varrho^{-1} M_h\mathbf{q}_h,\mathbf{q}_h)} \\[1mm] \nonumber
  & = & \overline{c}_\text{\tiny MD} \sup_{q_h \in V_h}  
        \frac{\displaystyle \left[ \int_\Omega
        \varrho^{1/4} \nabla v_h \cdot \varrho^{-1/4} \nabla q_h dx
        \right]^2}{\displaystyle
        \int_\Omega \varrho^{-1}[q_h(x)]^2 dx}\\[2mm] \nonumber
  & \le & \overline{c}_\text{\tiny MD} \sup_{q_h \in V_h}
        \frac{\|\varrho^{1/4} \nabla v_h\|_{L_2(\Omega)}^2
        \|\varrho^{-1/4} \nabla q_h\|_{L_2(\Omega)}^2}
        {\displaystyle \int_\Omega \varrho^{-1}[q_h(x)]^2 dx}\\ \nonumber
  & = & \overline{c}_\text{\tiny MD} \,
       \|\varrho^{1/4} \nabla v_h\|_{L_2(\Omega)}^2 \sup_{q_h \in V_h}  
       \frac{\displaystyle \sum\limits_{\tau_e \in \mathcal{T}_h}
       h^{-2} \int_{\tau_e} |\nabla q_h|^2 dx}
       {\displaystyle \int_\Omega \varrho^{-1}[q_h(x)]^2 dx}\\ \nonumber
  & \le & \overline{c}_\text{\tiny MD} \,
          \|\varrho^{1/4} \nabla v_h\|_{L_2(\Omega)}^2 \sup_{q_h \in V_h}
          \frac{\displaystyle \sum\limits_{\tau_e \in \mathcal{T}_h}
          h^{-4} \,c_\text{\tiny inv}^2 \int_{\tau_e} (q_h)^2 \, dx}
          {\displaystyle \sum_{\tau_e \in \mathcal{T}_h} h^{-4}
          \int_{\tau_e} (q_h)^2 dx} \\ \nonumber
  & = & c_\text{\tiny inv}^2 \, \overline{c}_\text{\tiny MD} \,
        \|\varrho^{1/4} \nabla v_h\|_{L_2(\Omega)}^2\\[1mm] \nonumber
  & = &c_\text{\tiny inv}^2 \, \overline{c}_\text{\tiny MD} \, 
        \sum\limits_{\tau_e \in \mathcal{T}_h} h^2 \int_{\tau_e}
        |\nabla v_h|^2 dx \\
  & \le & c_\text{\tiny inv}^4 \,\overline{c}_\text{\tiny MD} \,  \nonumber
        \sum\limits_{\tau_e \in \mathcal{T}_h} \int_{\tau_e} (v_h)^2 dx \\ 
  & = & c_\text{\tiny inv}^4 \, \overline{c}_\text{\tiny MD} \,
        (M_h \mathbf{v}_h,\mathbf{v}_h),
      \;\; \forall \mathbf{v}_h \in V_h,
\end{eqnarray}
where $ c_\text{\tiny inv} $ is the universal positive constant in the local inverse inequalities
\begin{equation}
\label{LLSY:Eqn:LocalInverseInequalities}
\|\nabla w_h\|_{L_2(\tau_e)}\le c_\text{\tiny inv} \, h_e^{-1} \,
\|w_h\|_{L_2(\tau_e)}\quad 
                   \forall w_h \in V_h, \; \forall \tau_e \in \mathcal{T}_h.
\end{equation}
Here the local mesh size $h_e$ can be replaced by the global mesh size $h$
since we assumed quasi-uniform and shape-regular mesh $\mathcal{T}_h$.
The local inverse inequalities \eqref{LLSY:Eqn:LocalInverseInequalities} 
can again be proved by mapping $\tau_e$ to the unit simplex $\tau$.
In this way the constant $ c_\text{\tiny inv} $ can even be computed explicitly 
in dependence of the mesh characteristics \cite{LLSY:ChenZhao:2013JCM}.
Therefore, we have just proved the spectral equivalence theorem
that is fundamental for the efficient solution of the spd mass-lumped
Schur-complement system 
\eqref{LLSY:Eqn:SchurComplementSystemMassLumpedCompact} 
by means of PCG iteration.

\begin{theorem}
\label{LLSY:Thm:SpectralEquivalenceDSD}
Let us assume that the mesh $\mathcal{T}_h$ is 
globally quasi-uniform with the global mesh-size $h$,
and $\varrho = h^4$. Then the spectral equivalence inequalities
\begin{equation}
 \label{LLSY:Eqn:SpectralEquivalenceDSD}
 \underline{c}_\text{\tiny SD} D_h \le \underline{c}_\text{\tiny SM} M_h 
 \le  S_h = \varrho K_h D_h^{-1} K_h + M_h \le 
 \overline{c}_\text{\tiny SM} M_h \le    \overline{c}_\text{\tiny SD} D_h,
\end{equation}
%
hold
with the spectral equivalence constants
\[
\underline{c}_\text{\tiny SM} = 1, \quad
\overline{c}_\text{\tiny SM} = c_\text{\tiny inv}^4
\overline{c}_\text{\tiny MD} + 1, \quad
\underline{c}_\text{\tiny SD} = \underline{c}_\text{\tiny MD} =
\lambda_\text{\tiny min} = \lambda_\text{\tiny min}(D_T^{-1}M_T) =
\frac{1}{d+2},
\]
and $\overline{c}_\text{\tiny SD} = \overline{c}_\text{\tiny MD}^2 c_\text{\tiny inv}^4 + \overline{c}_\text{\tiny MD}$,
where $\overline{c}_\text{\tiny MD} =  \lambda_\text{\tiny max} = \lambda_\text{\tiny max}(D_T^{-1}M_T) = 1$.
\end{theorem}
\begin{proof}
The  spectral equivalence inequalities \eqref{LLSY:Eqn:SpectralEquivalenceDSD} 
immediately follow from the
inequalities \eqref{LLSY:Eqn:SpectralEquivalenceDMD},
and \eqref{LLSY:Eqn:SpectralEquivalenceSD}.
\end{proof}
\begin{remark}
The spectral estimate \eqref{LLSY:Eqn:SpectralEquivalenceSD} can also be proved by Fourier analysis 
when one expands the vectors $\mathbf{v}_h$ into the orthonormal eigenvector basis corresponding 
to the eigenvalue problem $K_h \mathbf{e}_h = \lambda D_h  \mathbf{e}_h$ 
as it was done in {\rm
  \cite{LLSY:LangerLoescherSteinbachYang:2023}} for $D_h = M_h$.
In {\rm \cite{LLSY:LangerLoescherSteinbachYang:2023LSSC}},
we provide a rigorous analysis 
of the variable $L_2$ regularization with a technique that is different from 
the technique 
used for proving \eqref{LLSY:Eqn:SpectralEquivalenceSD}. 
We note that the latter technique can be used to analyse the case of constant and variable 
energy regularizations for state equations leading to non-symmetric fe
stiffness matrices $K_h$ such as 
convection-diffusion problems as well as parabolic and hyperbolic problems 
when using space-time finite element discretizations; see 
{\rm \cite{LLSY:LangerSteinbachYang:2022arXiv,LLSY:LoescherSteinbach:2022arXiv:2211.02562}}.
\end{remark}

\noindent
Now we can efficiently solve the mass-lumped Schur-complement system \eqref{LLSY:Eqn:SchurComplementSystemLumpedMass}
respectively \eqref{LLSY:Eqn:SchurComplementSystemMassLumpedCompact} by means of 
the PCG methods because, thanks to mass lumping, the matrix-vector multiplication $S_h * \mathbf{y}_h^k$ 
can be performed in asymptotically optimal complexity $O(n_h)$, and, 
at the same time,
the lumped mass matrix $D_h = \text{lump}(M_h)$ is a perfect preconditioner.
More precisely, let $\mathbf{y}_h^k \in \mathbb{R}^{n_h}$ be the $k$th PCG iterate.
Due to the spectral equivalence inequalities 
\eqref{LLSY:Eqn:SpectralEquivalenceDMD} and \eqref{LLSY:Eqn:SpectralEquivalenceDSD},
and the well-known convergence rate estimate for the PCG method 
(see, e.g., \cite[Chapter~13]{LLSY:Steinbach:2008Monograph}),
we can estimate the $L_2$ error $\|\hat y_{\varrho h}- y_{\varrho h}^k\|_{L_2(\Omega)}$ between
the fe functions 
$\hat y_{\varrho h}(x) = \sum_{i=1}^{n_h} y_i \varphi_i^h(x) \in V_h$ and
$y_{\varrho h}^k(x) = \sum_{i=1}^{n_h} y_i^k \varphi_i^h(x) \in V_h$ 
corresponding to the solution
$\mathbf{y}_h = (y_i)_{i=1,\ldots,n_h} \in \mathbb{R}^{n_h}$ of the Schur
complement system \eqref{LLSY:Eqn:SchurComplementSystemMassLumpedCompact} 
and the $k$-th PCG iterate
$\mathbf{y}_h^k = (y_i^k)_{i=1,\ldots,n_h} \in \mathbb{R}^{n_h}$, respectively,
as follows:
\begin{eqnarray}
\label{LLSY:Eqn:IterationErrorEstimate}
 \nonumber
  \|\hat y_{\varrho h} - y_{\varrho h}^k \|_{L_2(\Omega)} 
  &=& \| \mathbf{y}_h -  \mathbf{y}_h^k \|_{\mathbf{M}_h} 
      := (\mathbf{M}_h (\mathbf{y}_h -  \mathbf{y}_h^k), \mathbf{y}_h -  \mathbf{y}_h^k)^{1/2}    \\ \nonumber
  &\le& (\mathbf{S}_h (\mathbf{y}_h -  \mathbf{y}_h^k),\mathbf{y}_h -  \mathbf{y}_h^k)^{1/2}\\ \nonumber
  &=& \| \mathbf{y}_h -  \mathbf{y}_h^k \|_{\mathbf{S}_h} \le \, 2 \, q^k \, \| \mathbf{y}_h -  \mathbf{y}_h^0 \|_{\mathbf{S}_h} \\
  &\le&  2\, \overline{c}_{\text{\tiny SM}}^{1/2}\, q^k \, \| \mathbf{y}_h -  \mathbf{y}_h^0 \|_{\mathbf{M}_h}
        = 2\, \overline{c}_{\text{\tiny SM}}^{1/2}\, q^k \, \| \hat y_{\varrho h} - y_{\varrho h}^0 \|_{L_2(\Omega)},
\end{eqnarray} 
where $q =(\sqrt{\text{cond}_2(\mathbf{D}_h^{-1}\mathbf{S}_h)}-1)/(\sqrt{\text{cond}_2(\mathbf{D}_h^{-1}\mathbf{S}_h)}+1) < 1$,
and $\text{cond}_2(\mathbf{D}_h^{-1}\mathbf{S}_h) = 
        \lambda_\text{max}(\mathbf{D}_h^{-1}\mathbf{S}_h)/\lambda_\text{min}(\mathbf{D}_h^{-1}\mathbf{S}_h)$
denotes the spectral condition number of $\mathbf{D}_h^{-1}\mathbf{S}_h$ 
that can be bounded by the constant 
%
\begin{equation*}
\frac{\overline{c}_\text{\tiny SD}}{\underline{c}_\text{\tiny SD}} 
= \frac{\overline{c}_\text{\tiny MD}^2 c_\text{\tiny inv}^4 + \overline{c}_\text{\tiny MD}}{\underline{c}_\text{\tiny MD}}
= \frac{\lambda_\text{\tiny max}(D_T^{-1}M_T)^2 c_\text{\tiny inv}^4+ \lambda_\text{\tiny max}(D_T^{-1}M_T)}
            {\lambda_\text{\tiny min}(D_T^{-1}M_T)}
= (d+2)(c_\text{\tiny inv}^4 + 1)
\end{equation*}
that is independent of $h$.
Using the triangle inequality, 
the $L_2$-norm discretization error estimate 
\eqref{LLSY:Eqn:HsDiscretizationErrorEstimate} from Theorem~\ref{thm:hserror-lumped}, 
the $L_2$-norm iteration error estimate
\eqref{LLSY:Eqn:IterationErrorEstimate}, the inequality
\begin{equation}
 \| \hat y_{\varrho h}\|_{L_2(\Omega)} \le \| y_d\|_{L_2(\Omega)}
\end{equation}
that follow from \eqref{LLSY:Eqn:DVF-optimality-system-scaled-lumped1}-\eqref{LLSY:Eqn:DVF-optimality-system-scaled-lumped2}   
when we choose the test functions $q_h = \hat p_{\varrho h}$ and $v_h = \hat y_{\varrho h}$,
we finally arrive  at $L_2$-norm estimate between desired state $y_d$ and the $k$th PCG iterate
$y_{\varrho h}^k$ computed by the PCG method:
\begin{eqnarray}
\label{LLSY:eqn:L2IterationError}
\nonumber
  \| y_d - y_{\varrho h}^k \|_{L_2(\Omega)}
  &\le& \| y_d - \hat y_{\varrho h} \|_{L_2(\Omega)} + \|  \hat y_{\varrho h} - y_{\varrho h}^k \|_{L_2(\Omega)}\\\nonumber
  &\le& c \, h^s \, \|y_d\|_{H^s(\Omega)} +
        2\, \overline{c}_{\text{\tiny SM}}^{1/2}\, q^k \, \| \hat y_{\varrho h} - y_{\varrho h}^0 \|_{L_2(\Omega)}   \\\nonumber
  &\le& h^s \Big( c \,\|y_d\|_{H^s(\Omega)} 
        + 2\, \overline{c}_{\text{\tiny SM}}^{1/2} \, \| \hat y_{\varrho h} - y_{\varrho h}^0 \|_{L_2(\Omega)} \Big)\\
  &\le& h^s \Big( c \,\|y_d\|_{H^s(\Omega)} 
        + 2\, \overline{c}_{\text{\tiny SM}}^{1/2} \, \| y_d\|_{L_2(\Omega)} \Big) = h^s c(y_d)
\end{eqnarray}
provided that $q^k \le h^s$ and that the initial guess $y_{\varrho h}^0$ is chosen to be zero. 
Therefore, $k \ge  \ln h^{-s} / \ln q^{-1}$ ensures that the PCG computes an approximation
$y_{\varrho h}^k$ to the desired state $y_d$ that differs from $y_d$ 
in the same order $O(h^s)$ as the discretization error $\| y_d - \hat y_{\varrho h} \|_{L_2(\Omega)}$
in the $L_2$ norm. 
Moreover, this can be done with $O(n_h  \ln h^{-1}) = O(h^{-d} \ln h^{-1}) $ arithmetical operations,
i.e. the complexity is asymptotically optimal up to the logarithmical 
factor $\ln h^{-1}$.

This logarithmical factor can be avoided in a nested iteration setting 
on a sequence of refined (nested) meshes. Indeed, let us consider a sequence of uniformly 
refined meshes $\mathcal{T}_\ell = \mathcal{T}_{h_\ell}$ with the mesh size $h_\ell$ 
and the optimally balanced regularization parameter $\varrho_\ell = h_\ell^4$, $\ell=1,\ldots,L$,
where $h_\ell = h_{\ell-1}/2$, $\ell=2,\ldots,L$. 
Thus, the coarsest mesh corresponds to the subindex $1$, whereas the finest mesh is related to $L$.
On every mesh $\mathcal{T}_\ell$,  $\ell = 1,\ldots,L$, we have to solve the mass-lumped Schur-complement system 
\eqref{LLSY:Eqn:SchurComplementSystemMassLumpedCompact}
that we now write in form:
find $\mathbf{y}_\ell = \mathbf{y}_{h_\ell} \in \mathbb{R}^{n_\ell}=\mathbb{R}^{n_{h_\ell}}$ such that
\begin{equation}
 \label{LLSY:Eqn:SchurComplementSystemMassLumpedCompact_l}
		S_\ell \mathbf{y}_\ell = \mathbf{y}_{d\ell}
\end{equation}
where 
$S_\ell = \varrho_\ell K_\ell D_\ell^{-1} K_\ell + M_\ell$, $K_\ell = K_{h_\ell}$. $D_\ell = D_{h_\ell}$,
$M_\ell = M_{h_\ell}$, $\mathbf{y}_{d\ell}=\mathbf{y}_{dh_\ell}$, and $\varrho_\ell = h_\ell^4$. 

%
Now the {\it nested iteration algorithm} works as follows. 
First we solve the coarse-mesh problem \eqref{LLSY:Eqn:SchurComplementSystemMassLumpedCompact_l}, $\ell =1$,
sufficiently accurate. More precisely, we compute an iterate $\mathbf{y}_1^{k_1} \in \mathbb{R}^{n_1}$
corresponding to the fe function  $y_1^{k_1} = y_{\varrho_1 h_1}^{k_1} \in V_1 = V_{h_1}$ 
(short: $\mathbf{y}_1^{k_1}\leftrightarrow y_1^{k_1}$) such that
\begin{equation}
\|y_d - y_1^{k_1}\|_{L_2(\Omega)} \le h_1^s c(y_d).
\end{equation}
Due to \eqref{LLSY:eqn:L2IterationError}, this can be done with $k_1 \ge \ln h_1^{-s} / \ln q^{-1}$
PCG iterations starting with $y_1^{0} =0$.
Now, let us assume that, on level $\ell-1 \in \{1,\ldots,L-1\}$,
the iterate $y_{\ell-1}^{k_{\ell-1}} \in V_{\ell-1}$ fulfills the estimate
\begin{equation}
\label{LLSY:Eqn:L2NestedIterationErrorIV}
\|y_d - y_{\ell-1}^{k_{\ell-1}}\|_{L_2(\Omega)} \le h_{\ell-1}^s c(y_d).
\end{equation}
Let 
$y_\ell^0 = I_{\ell-1}^\ell y_{\ell-1}^{k_{\ell-1}}$ 
be the affine-linear interpolate of $y_{\ell-1}^{k_{\ell-1}}$.
We note that $y_\ell^0 = I_{\ell-1}^\ell y_{\ell-1}^{k_{\ell-1}} = y_{\ell-1}^{k_{\ell-1}} \in V_{\ell-1} \subset V_\ell$
since the meshes are nested.
Then we get the estimate 
\begin{eqnarray}
\label{LLSY:Eqn:L2NestedIterationErrorl_1}
\nonumber
 \| y_d - y_\ell^{k_\ell} \|_{L_2(\Omega)}
  &\le& \| y_d - \hat y_\ell \|_{L_2(\Omega)} + \|  \hat y_\ell -  y_\ell^{k_\ell}\|_{L_2(\Omega)}\\ 
  &\le& c h_\ell^s \|y_d\|_{H^s(\Omega)} +    
            2\, \overline{c}_{\text{\tiny SM}}^{1/2}\, q^{k_\ell} \, \| \hat y_\ell - y_\ell^0 \|_{L_2(\Omega)}, 
\end{eqnarray}
where $\hat y_\ell = \hat y_{\varrho_\ell h_\ell} \in V_\ell$ is the exact state solution of the 
finite element scheme 
\eqref{LLSY:Eqn:DVF-optimality-system-scaled-lumped1}-\eqref{LLSY:Eqn:DVF-optimality-system-scaled-lumped2} 
corresponding to the solution $\mathbf{y}_\ell \in \mathbb{R}^{n_\ell}$ of the mass-lumped Schur-complement system 
\eqref{LLSY:Eqn:SchurComplementSystemMassLumpedCompact_l}.
Now using $y_\ell^0 = I_{\ell-1}^\ell y_{\ell-1}^{k_{\ell-1}}$, the triangle inequality, Theorem~\ref{thm:hserror-lumped},
and estimate \eqref{LLSY:Eqn:L2NestedIterationErrorIV},  we can continue to estimate the last 
term in \eqref{LLSY:Eqn:L2NestedIterationErrorl_1} as follows:
\begin{eqnarray}
\label{LLSY:Eqn:L2NestedIterationErrorl_2}
\nonumber
 \|  \hat y_\ell - y_\ell^0\|_{L_2(\Omega)}
  &\le& \| \hat y_\ell - y_d\|_{L_2(\Omega)} + \|y_d - I_{\ell-1}^\ell y_{\ell-1}^{k_{l-1}}\|_{L_2(\Omega)}\\ \nonumber
  &\le& \| \hat y_\ell - y_d\|_{L_2(\Omega)} + \|y_d -  y_{\ell-1}^{k_{l-1}}\|_{L_2(\Omega)}\\ \nonumber
  &\le& c h_\ell^s \|y_d\|_{H^s(\Omega)} +  h_{\ell-1}^s c(y_d) \\
  &\le& h_\ell^s (c \|y_d\|_{H^s(\Omega)} +  2^s c(y_d))       
\end{eqnarray}
Inserting \eqref{LLSY:Eqn:L2NestedIterationErrorl_2} into \eqref{LLSY:Eqn:L2NestedIterationErrorl_1},
we get 
\begin{eqnarray}
\label{LLSY:Eqn:L2NestedIterationErrorl}
\nonumber
 \| y_d - y_\ell^{k_\ell} \|_{L_2(\Omega)}
     &\le& h_\ell^s \left[ c \|y_d\|_{H^s(\Omega)} +  
            2\, \overline{c}_{\text{\tiny SM}}^{1/2}\, q^{k_\ell} (c \|y_d\|_{H^s(\Omega)} +  2^s c(y_d))\right] \\
     &\le&  h_{\ell}^s c(y_d)      
\end{eqnarray}
provided that $q^{k_\ell}(c \|y_d\|_{H^s(\Omega)} +  2^s c(y_d)) \le \|y_d\|_{L_2(\Omega)}$.
The latter inequality is ensured if we performed not more than
\begin{equation}
\label{LLSY:Eqn:NestedIterationNumber_k*}
 k_\ell = k_* \ge \ln (q(y_d)^{-1}) / \ln q^{-1}
\end{equation}
nested iterations, where 
$
q(y_d) = \|y_d\|_{L_2(\Omega)}/((1+2^s)c \|y_d\|_{H^s(\Omega)} + 2^{1+s}\|y_d\|_{L_2(\Omega)}) < 1.
$
Here we exclude the trivial case that $y_d=0$.
Therefore, we have proved the following nested iteration theorem by induction.
\begin{theorem}
\label{LLSY:Thm:NestedIteration}
If the coarse mesh problem on level $l=1$ is solved by $k_1$ PCG iterations with the initial guess 
$\mathbf{y}_1^0 = \mathbf{0}_1$
such that \eqref{LLSY:Eqn:L2NestedIterationErrorl_1} holds, and if $k_*$ nested PCG iterations 
are used on all nested levels $\ell = 2,\ldots,L$, i.e. $k_2=\ldots=k_L= k_*$ defined by \eqref{LLSY:Eqn:NestedIterationNumber_k*},
then the last iterate $y_L^{k_L} \leftrightarrow \mathbf{y}_L^{k_L}$ on the finest level $\ell = L$
differs from given desired state $y_d$ in the order of the discretization error $O(h_L^s)$ 
with respect to the $L_2(\Omega)$ norm. More precisely, we get the estimate 
\begin{equation}
\label{LLSY:Eqn:L2NestedIterationErrorL}
\|y_d - y_L^{k_L}\|_{L_2(\Omega)} \le h_L^s c(y_d).
\end{equation}
The computation of $y_L^{k_L} \leftrightarrow \mathbf{y}_L^{k_L}$ requires not more 
than $O(n_L) = O(h_L^{-d})$ arithmetical operations and memory,
i.e. the nested iteration procedure proposed is asymptotically optimal.
\end{theorem}
\begin{proof}
The proof of estimate \eqref{LLSY:Eqn:L2NestedIterationErrorL} follows from above by induction.
The complexity analysis is based on simple use of the geometric series.
\end{proof}

%


We stop the nested iteration process as soon as we arrive at some desired relative 
accuracy $\varepsilon \in (0,1)$ such that
\begin{equation}
\label{LLSY:Eqn:DesiredL2Error}
 \|y_d - y_L^{k_L}\|_{L_2(\Omega)} \le \varepsilon \, \|y_d\|_{L_2(\Omega)}.
\end{equation}
The apriori estimate \eqref{LLSY:Eqn:L2NestedIterationErrorL} immediately yields that 
estimate \eqref{LLSY:Eqn:DesiredL2Error} is guaranteed when 
$ch_L^s \|y_d\|_{H^s(\Omega)} \le \varepsilon \, \|y_d\|_{L_2(\Omega)}$,
but in practice we directly check \eqref{LLSY:Eqn:DesiredL2Error} 
because all quantities are computable.

We will also stop the nested iteration if the cost for the control $u_\ell^{k_\ell}$ 
becomes too large, where $u_\ell^{k_\ell}\leftrightarrow \mathbf{u}_\ell^{k_\ell}$ 
is computed from the fe state equation
\begin{equation*}
\mathbf{u}_\ell^{k_\ell} = - \varrho_\ell^{-1}\mathbf{p}_\ell^{k_\ell} = D_\ell^{-1} K_\ell \mathbf{y}_\ell^{k_\ell}
\end{equation*}
More precisely, let $c_u > 0$ be a given threshold for the control cost that we are willing to pay. 
Then we stop the nested iteration if
\begin{equation*}
\|u_\ell^{k_\ell}\|_{L_2(\Omega)} = (M_\ell \mathbf{u}_\ell^{k_\ell},\mathbf{u}_\ell^{k_\ell})
\le \overline{c}_\text{\tiny MD} (D_\ell \mathbf{u}_\ell^{k_\ell},\mathbf{u}_\ell^{k_\ell})
=  \overline{c}_\text{\tiny MD} (K_\ell \mathbf{y}_\ell^{k_\ell},\mathbf{y}_\ell^{k_\ell})
\le c_u,
\end{equation*}
but $\|u_{\ell+1}^{k_{\ell+1}}\|_{L_2(\Omega)} > c_u$, where  $\overline{c}_\text{\tiny MD}=1$. 
Then we set $L=\ell$.

Now we may proceed with {\it cascadic nested iteration}
freezing the cost (regularization) parameter $\varrho_L = h_L^4$ 
and refining the mesh  only , i.e.
\begin{equation}
 \varrho_{\ell+1} = \varrho_L = h_L^4 = \text{const.} \; \; \mbox{and} \; \; h_{\ell+1}= h_\ell / 2 
                        \;\mbox{for}\; \ell = L,\ldots,L+J-1,
\end{equation}
in order to improve the approximation of the control. 
We note that this further mesh refinement will not improve the approximation 
to the desired state $y_d$ 
since this error 
is defined by the frozen cost parameter $\varrho_L$. 
%
%
Since we only use a few additional levels for the improvement of the 
control, we can proceed with the PCG 
preconditioned by 
$D_\ell$ 
as before
as nested iteration, but replacing $\overline{c}_\text{\tiny SD}$ by 
$\overline{c}_{\text{\tiny SD},\ell} = \overline{c}_\text{\tiny SD} 2^{4(\ell - L)}$
for $\ell = L+1,\ldots,L+J$. 
If we want to add many  levels, i.e. $J>>1$, then we 
may use some 
cascadic full multigrid Schur complement iteration using level $L$ as coarse mesh
and $y_{L+1}^0 = I_L^{L+1} y_L^{k_L} = y_L^{k_L} \in V_L \subset V_{L+1}$ as initial guess;
see \cite{LLSY:Hackbusch:1985Book,LLST:Braess:2007a} 
for $L_2$ convergent multigrid methods. 



%
\section{Numerical Results}
\label{LLSY:sec:NumericalResults}	
In our numerical experiments, we consider the discontinuous
desired state 
\begin{equation*}
  y_d=
  \begin{cases}
    1 & \textup{ in }\; (0.25, 0.75)^3,\\
    0 & \textup{ in }\; \overline{\Omega}\setminus(0.25, 0.75)^3,
  \end{cases}
\end{equation*}
in the computational domain $\Omega=(0, 1)^3 \subset \mathbb{R}^{d=3}$.
This discontinuous desired state $y_d$ does not belong to $Y=H^1_0(\Omega)$,
and has a rather low Sobolev regularity. 
More precisely, $y_d \in H^{1/2 - \varepsilon}(\Omega)$ for any $\varepsilon > 0$.
This discontinuous target has been utilized
in the work \cite{LLSY:LangerLoescherSteinbachYang:2023, LLSY:LangerSteinbachYang:2022b,
  LLSY:NeumuellerSteinbach:2021a} in both the cases of $L_2$ and energy ($H^{-1}$)
regularization for distributed elliptic optimal control problems.
So, we can easily compare the numerical results presented below for 
the mass-lumping discretization of the control term 
in the reduced optimality system 
with those of the $L_2$ regularization without mass lumping and the $H^{-1}$ regularization.

We decompose the domain $\Omega=(0,1)^3$ into uniformly refined tetrahedral
elements $\tau_e$,
and start with an initial mesh that contains $384$ tetrahedral
elements and $125$ vertices, leading to the mesh size $h=2^{-2}$. From such a
mesh, we make successive refinements on the levels 
$\ell = 1,...,8$.
On the finest level $\ell=L=8$, we have $135,005,697$ vertices, $h=2^{-9}=1.9531$e$-3$, and
$\varrho=h^4=2^{-36}=1.4552$e$-11$. Further, we run tests on the adaptively
refined meshes, in which we have employed the standard red-green refinement of
tetrahedral elements, and we have chosen the locally varying regularization parameter
$\varrho_\tau=h_e^4$ on each tetrahedral element $\tau_e$. 
The adaptive procedure is simply based on the localization of the  
error $\|y_d - {\tilde y}_\ell\|_{L_2(\Omega)}$ that is explicitly computable 
for any known fe approximation ${\tilde y}_\ell$ to the given desired state $y_d$;
see \cite{LLSY:LangerLoescherSteinbachYang:2022b} for a detailed description.

%

As described in Section \ref{LLSY:sec:NestedPCGIteration}, thanks to the replacement 
of the mass matrix  $M_h$ by its diagonal approximation $D_h = \text{lump}(M_h)$, 
we can efficiently solve the spd mass-lumped Schur-complement system 
\eqref{LLSY:Eqn:SchurComplementSystemMassLumpedCompact} 
by means of the PCG preconditioned by $D_h$. 
We first use the initial guess $\mathbf{y}_\ell^0= \mathbf{0}$, and terminate the iteration 
as soon as the preconditioned residual is reduced by a factor $10^6$.
The  number of PCG iterations (Its)and the computational time (Time) in seconds (s)
are provided in Table \ref{tab:solver} for both uniform and adaptive refinements. 

Therein, we observe the robustness of  our proposed preconditioner for
(\ref{LLSY:Eqn:SchurComplementSystemMassLumpedCompact})
with respect to both the mesh size and local adaptivity under the choice of
$\varrho_\tau=h_e^4$. We only see slightly more iterations for the adaptive
refinements in comparison to the uniform refinements. 
\begin{table}
%
{\small
    \begin{tabular}{|l|rlr|rlr|}
      \hline
      \multirow{2}{*}{$\ell$}&\multicolumn{3}{c|}{Adaptive} &
      \multicolumn{3}{c|}{Uniform} \\  \cline{2-7}
      &\#Dofs& error &Its (Time)& \#Dofs &error&Its (Time) \\
      \hline
      $1$&$125$ &$3.26$e$-1$&$10$ ($6.3$e$-4$)&$125$&$3.26$e$-1$ &$10$ ($6.4$e$-4$)\\
      $2$&$223$&$2.35$e$-1$&$62$ ($6.5$e$-3$)&$729$&$2.25$e$-1$&$55$ ($1.8$e$-2$)\\
      $3$&$1,044$&$1.86$e$-1$&$106$ ($5.5$e$-2$)&$4,913$&$1.59$e$-1$&$79$ ($1.9$e$-1$)\\
      $4$&$4,548$&$1.32$e$-1$&$123$ ($2.8$e$-1$)&$35,937$&$1.12$e$-1$&$85$ ($1.7$e$-0$)\\
      $5$&$10,524$&$1.05$e$-1$&$116$ ($6.3$e$-1$)&$274,625$&$7.96$e$-2$&$81$ ($2.3$e$+1$)\\
      $6$&$25,807$&$8.35$e$-2$&$113$ ($1.6$e$-0$)&$2,146,689$&$5.62$e$-2$&$74$ ($1.8$e$+2$)\\
      $7$&$91,520$&$6.03$e$-2$&$100$ ($5.7$e$-0$)&$16,974,593$&$3.97$e$-2$&$68$ ($1.4$e$+3$)\\
      $8$&$118,334$&$5.62$e$-2$&$102$ ($7.7$e$-0$)&$135,005,697$&$2.81$e$-2$&$66$ ($1.3$e$+4$)\\
      $9$&$432,195$&$4.08$e$-2$&$93$ ($3.5$e$+1$)&&& \\
      ${10}$&$473,638$&$3.97$e$-2$&$95$ ($6.3$e$+1$)&&&\\
      ${11}$&$1,843,740$&$2.84$e$-2$&$91$ ($2.5$e$+2$)&&&\\
      ${12}$&$1,937,983$&$2.79$e$-2$&$92$ ($3.2$e$+2$)&&&\\
      ${13}$&$7,681,306$&$1.99$e$-2$&$91$ ($6.3$e$+2$)&&&\\
      ${14}$&$7,922,574$&$1.96$e$-2$&$93$ ($1.0$e$+3$)&&&\\
      ${15}$&$31,496,575$&$1.39$e$-2$&$83$ ($3.6$e$+3$)&&&\\
      ${16}$&$32,000,845$&$1.38$e$-2$&$84$ ($5.4$e$+3$)&&&\\
      ${17}$&$127,607,911$&$9.84$e$-3$&$68$ ($1.6$e$+4$)&&&\\
      \hline
    \end{tabular}
\caption{Comparison of the PCG iterations (Its) and computational time (Time) in
  seconds (indicated in the parentheses) for
  solving
  (\ref{LLSY:Eqn:SchurComplementSystemMassLumpedCompact}) 
  for
  both adaptive and
  uniform refinements using the non-nested iterations,
  where error = $\|y_d - y_\ell^{k_\ell}\|_{L_2(\Omega)}$.
  } 
\label{tab:solver}
}
\end{table}
%

As shown in the theoretical part, solving the 
Schur complement equation with the lumped mass does not
deteriorate the convergence of our finite element approximation. This is confirmed
in our numerical experiments. The comparison of convergence on both uniform and adaptive refinements
is given in Figure \ref{fig:errorcomp}.
We observe the convergence rate
$h^{0.5}$ for the uniform refinement as predicted by Theorem\ref{thm:hserror-lumped}, and a much better convergence rate
$h^{0.75}$ for the adaptive refinements;
  see \cite{LLSY:LangerLoescherSteinbachYang:2022b} for the case of variable energy regularization.
There one can also find an explanation of the convergence rate that can be achieved via 
this adaptive procedure.

%
\begin{figure}
  \centering
    \includegraphics[width=1.0\textwidth]{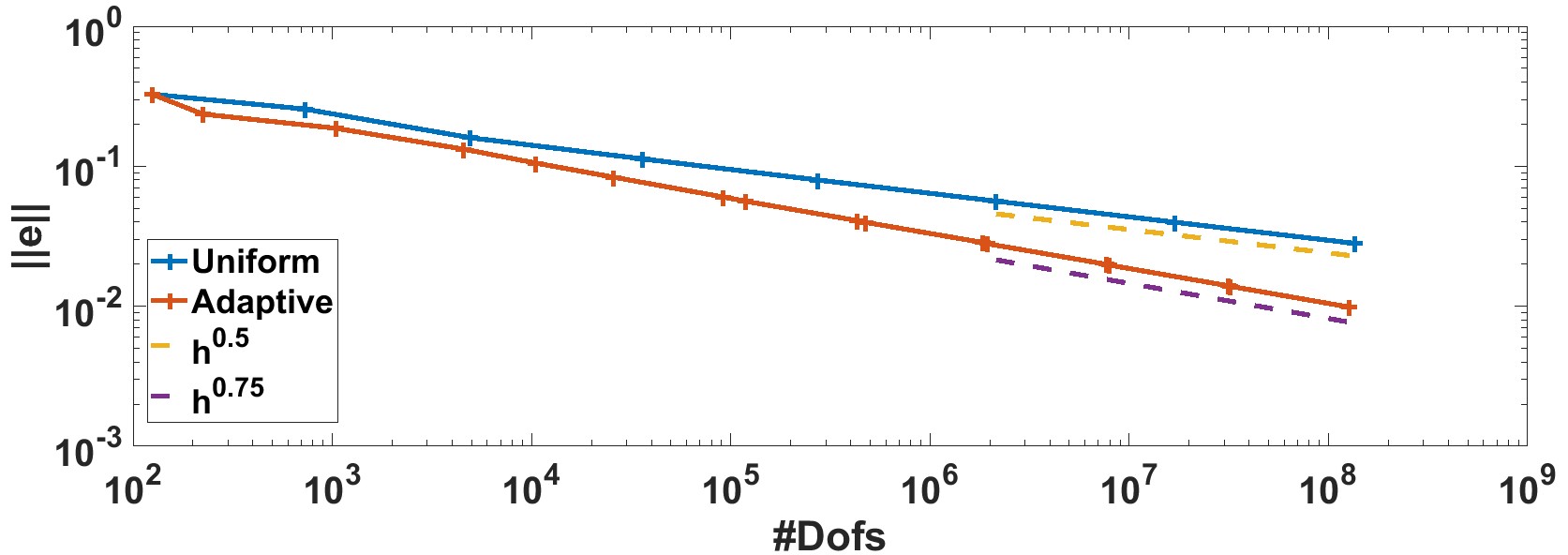}
    \caption{Comparison of the convergence history obtained from non-nested iterations for uniform and adaptive refinements,
      where 
        $\|e\| = \|y_d - y_\ell^{k_\ell}\|_{L_2(\Omega)}$.
      }
    \label{fig:errorcomp}
\end{figure}

In order to further reduce the computational cost, we utilize 
nested PCG iterations as described in Section~\ref{LLSY:sec:NestedPCGIteration}.
Here, on the coarsest level 
$\ell = 1$,
we run the
PCG iterations until the relative preconditioned residual reaches $10^{-6}$. 
On the refined levels 
$\ell = 2,3,...$,
we have utilized an adaptive tolerance
\begin{equation}\label{eq:adapttol}
  \alpha\left[n_\ell/n_{\ell-1}\right]^{-\frac{\beta}{3}},\; \ell =2,3,\dots,
\end{equation}
for the relative preconditioned residual, with $\alpha$ being a scaling factor, $\beta=0.5$ and
$0.75$ for the uniform and adaptive refinement, respectively,  
and
$n_\ell$
the number of degrees of freedom (\#Dofs) 
on the mesh level 
$\ell = 1,2,\ldots$ .
The solution on the 
level $\ell-1$
is used as an initial guess for the PCG iteration on the next finer level $\ell$.
The reduced number of nested iterations (Its) on both uniform and
adaptive refinements is given in Table \ref{tab:solver_nested}, where we
have chosen $\alpha=0.5$ and $\alpha=1$ for the adaptive and uniform
refinement, respectively. From this, we easily see much fewer iteration
numbers and significantly less computational time in seconds in comparison with the case of
non-nested iterations as shown in Table \ref{tab:solver}, without loss of
accuracy of the numerical approximations; see Figure~\ref{fig:errorcomp_nested} 
for a comparison of convergence history  
for
both uniform and
adaptive refinements using the nested iterations.  
\begin{table}[h]
{\small
    \begin{tabular}{|l|rlr|rlr|}
      \hline
      \multirow{2}{*}{$\ell$}&\multicolumn{3}{c|}{Adaptive} &
      \multicolumn{3}{c|}{Uniform} \\  \cline{2-7}
      &\#Dofs& error &Its (Time)& \#Dofs & error  &Its (Time)\\
      \hline
      $1$&$125$ &$3.26$e$-1$&$10$ ($6.3$e$-4$)&$125$&$3.26$e$-1$ &$10$ ($6.3$e$-4$)\\
      $2$&$223$&$3.30$e$-1$&$1$ ($2.6$e$-4$)&$729$&$2.27$e$-1$&$8$ ($2.9$e$-3$)\\
      $3$&$1,067$&$1.84$e$-1$&$19$ ($1.1$e$-2$)&$4,913$&$2.25$e$-1$&$1$ ($4.6$e$-3$)\\
      $4$&$4,705$&$1,28$e$-1$&$13$ ($3.3$e$-2$)&$35,937$&$1.08$e$-1$&$9$ ($1.9$e$-1$)\\
      $5$&$15,368$&$1.00$e$-1$&$17$ ($1.4$e$-1$)&$274,625$&$8.22$e$-2$&$8$ ($1.5$e$-0$)\\
      $6$&$30,996$&$8.45$e$-2$&$17$ ($4.0$e$-1$)&$2,146,689$&$5.60$e$-2$&$9$ ($1.4$e$+1$)\\
      $7$&$94,176$&$6.30$e$-2$&$19$ ($1.3$e$-0$)&$16,974,593$&$3.98$e$-2$&$9$ ($2.1$e$+2$)\\
      $8$&$129,760$&$5.68$e$-2$&$18$ ($1.7$e$-0$)&$135,005,697$&$2.81$e$-2$&$9$ ($2.2$e$+3$)\\
      $9$&$440,572$&$4.18$e$-2$&$17$ ($1.2$e$+1$)&&& \\
      ${10}$&$488,124$&$4.03$e$-2$&$17$ ($1.3$e$+1$)&&& \\
      ${11}$&$1,860,339$&$2.90$e$-2$&$18$ ($6.1$e$+1$)&&&\\
      ${12}$&$1,958,388$&$2.85$e$-2$&$16$ ($5.9$e$+1$)&&&\\
      ${13}$&$7,254,384$&$2.06$e$-2$&$18$ ($2.6$e$+2$)&&&\\
      ${14}$&$7,408,106$&$2.04$e$-2$&$16$ ($2.1$e$+2$)&&&\\
      ${15}$&$29,094,073$&$1.47$e$-2$&$17$ ($6.9$e$+2$)&&&\\
      ${16}$&$29,682,531$&$1.44$e$-2$&$16$ ($7.6$e$+2$)&&&\\
      ${17}$&$116,229,104$&$1.04$e$-2$&$16$ ($3.7$e$+3$)&&&\\
      \hline
    \end{tabular}
}
\caption{Comparison of the PCG iterations (Its) and computational time (Time) in
  seconds (indicated in the parentheses) 
  for solving 
  (\ref{LLSY:Eqn:SchurComplementSystemMassLumpedCompact})
  for
  both adaptive ($\alpha=0.5$,
  $\beta=0.75$) and uniform ($\alpha=1$, $\beta=0.5$) refinements using
  the nested iteration approach,
  where error = $\|y_d - y_\ell^{k_\ell}\|_{L_2(\Omega)}$
  } 
  \label{tab:solver_nested}
\end{table}
%
%
\begin{figure}
  \centering
    \includegraphics[width=1.0\textwidth]{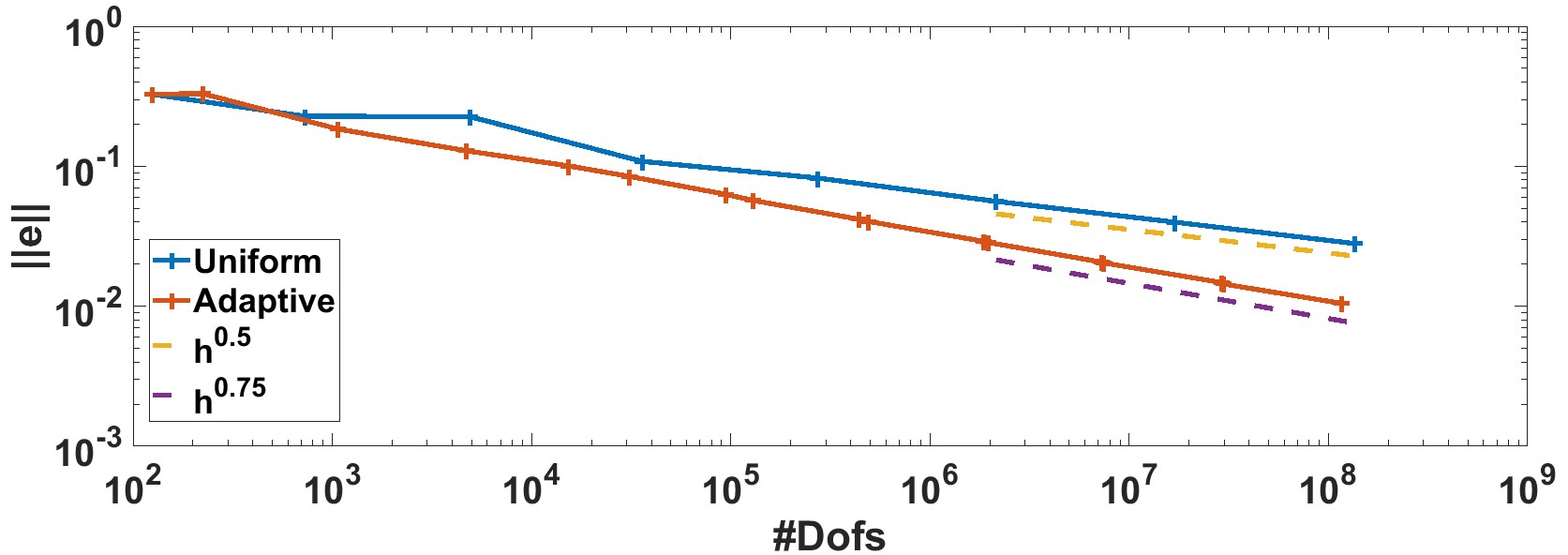}
    \caption{
     Comparison of the convergence history obtained from nested iterations for uniform and adaptive refinements, where
     $\|e\| = \|y_d - y_\ell^{k_\ell}\|_{L_2(\Omega)}$.
      }
    \label{fig:errorcomp_nested}
\end{figure}

  Another approach to reduce the computational time, especially, in the case of uniform refinement 
is the parallelization of the PCG solver. The parallelization of the conjugate gradient algorithm is now a standard 
procedure \cite{LLST:DouglasHaaseLanger:2003Book}. The crucial point is always the preconditioner.
It is clear that the parallelization of a diagonal preconditioner is much easier 
than the parallelization of a multigrid preconditioner. More precisely, 
the parallelization of a diagonal preconditioner such as $D_h$ is trivial.
%
For parallel performance studies, we have utilized the open source MFEM\footnote{https://mfem.org/}.
We observe from the diagonals of Table~\ref{tab:parallel_solver_nonnested},
e.g. from level $7$ with $16$ cores to level $8$ with $512$ cores (always factor $8$), 
almost constant time, i.e. a good weak scaling behavior, 
whereas the horizontal lines show an almost perfect strong scaling.
The latter one is also illustrated in Figure~\ref{fig:cpu_nonnested} 
for $\ell = 7$ and $\ell = 8$ corresponding to $16,974,593$ and $135,005,697$ Dofs,
respectively. The largest problem with $135,005,697$ Dofs can be solved in $6.4$ seconds 
using $512$ cores. 
Similar scaling behaviors are also observed for the nested iterations approach; 
see Table~\ref{tab:parallel_solver_nested} and Figure~\ref{fig:cpu_nested}.
The computational time in
seconds (s) using nested iterations is further reduced by a factor of about $7$ in
comparison with the non-nested iterations. 
Using $512$ cores, the largest problem with $135,005,697$ Dofs is solved in $1$ second. 
Finally, we made some performance tests for the
adaptive refinement using the nested iteration setting. 
The results are given in Table~\ref{tab:parallel_solver_adaptive_nested}. 
We observe relatively good scaling in this case as well. 
Here, we have used the non-conforming simplicial complex and load balance from the open source MFEM.

We note that we used different computers and different codes for the single-core and 
parallel computations. More precisely, we used the shared-memory computer 
MACH2\footnote{https://www3.risc.jku.at/projects/mach2/},
that provides a big memory,
and the distributed-memory computer RADON1\footnote{https://www.oeaw.ac.at/ricam/hpc} 
for the single-core and 
parallel computations, respectively. 

\begin{table}
  {\small
    \begin{tabular}{|c|c|c|c|c|c|c|}
    \hline
    \multirow{2}{*}{$\ell$}&\multicolumn{6}{c|}{\#Cores}\\  \cline{2-7}
    &16&32&64&128&256&512\\
    \hline
    4&$85$ ($4.0$e$-2$)&-&-&-&-&-\\ 
    5&$84$ ($3.4$e$-1$)&$84$ ($1.6$e$-1$)&$84$ ($6.4$e$-2$)&-&-&-\\
    6&$82$ ($2.9$e$-0$)&$82$ ($1.4$e$-0$)&$82$ ($7.3$e$-1$)&$82$ ($3.7$e$-1$)&$82$ ($1.8$e$-1$)&$82$ ($8.0$e$-2$)\\
    7&$80$ ($2.5$e$+1$)&$80$ ($1.2$e$+1$)&$80$ ($6.3$e$-0$)&$80$ ($3.0$e$-0$)&$80$ ($1.5$e$-0$)&$80$ ($8.3$e$-1$)\\
    8&-&-&$77$ ($5.1$e$+1$)&$77$ ($2.5$e$+1$)&$77$ ($1.3$e$+1$)& $77$ ($6.4$e$-0$) \\ \hline
    \end{tabular}
    \caption{Parallel performance on a distributed computer system for uniform refinement and non-nested iterations.
    }
    \label{tab:parallel_solver_nonnested}
  }
\end{table}
\begin{table}
  {\small
    \begin{tabular}{|c|c|c|c|c|c|c|c|}
      \hline\multirow{2}{*}{$\ell$}&\multicolumn{6}{c|}{\#Cores}\\  \cline{2-7}
      &16&32&64&128&256&512\\
      \hline
      4&$9$ ($5.8$e$-3$)&-&-&-&-&-\\ 
      5&$9$ ($5.0$e$-2 $)&$9$ ($2.7$e$-2 $)&$9$ ($1.1$e$-2$)&-&-&-\\
      6&$8$ ($3.5$e$-1$)&$8$ ($1.8$e$-1$)&$8$ ($9.4$e$-2$)&$8$ ($5.2$e$-2$)&$8$ ($2.8$e$-2$)&$8$ ($1.2$e$-2$)\\
      7&$9$ ($3.1$e$-0$)&$9$ ($1.6$e$-0$)&$9$ ($8.1$e$-1$)&$9$ ($4.2$e$-1$)&$9$ ($2.2$e$-1$)&$9$ ($1.2$e$-1$)\\
      8&-&-&$11$ ($7.9$e$-0$)&$11$ ($4.0$e$-0$)&$11$ ($2.0$e$-0$)&$11$ ($1.0$e$-0$) \\ \hline
    \end{tabular}
    \caption{Parallel performance on a distributed computer system for uniform refinement and nested iterations.}
    \label{tab:parallel_solver_nested}
  }
\end{table}
\begin{figure}
  \centering
    \includegraphics[width=1.0\textwidth]{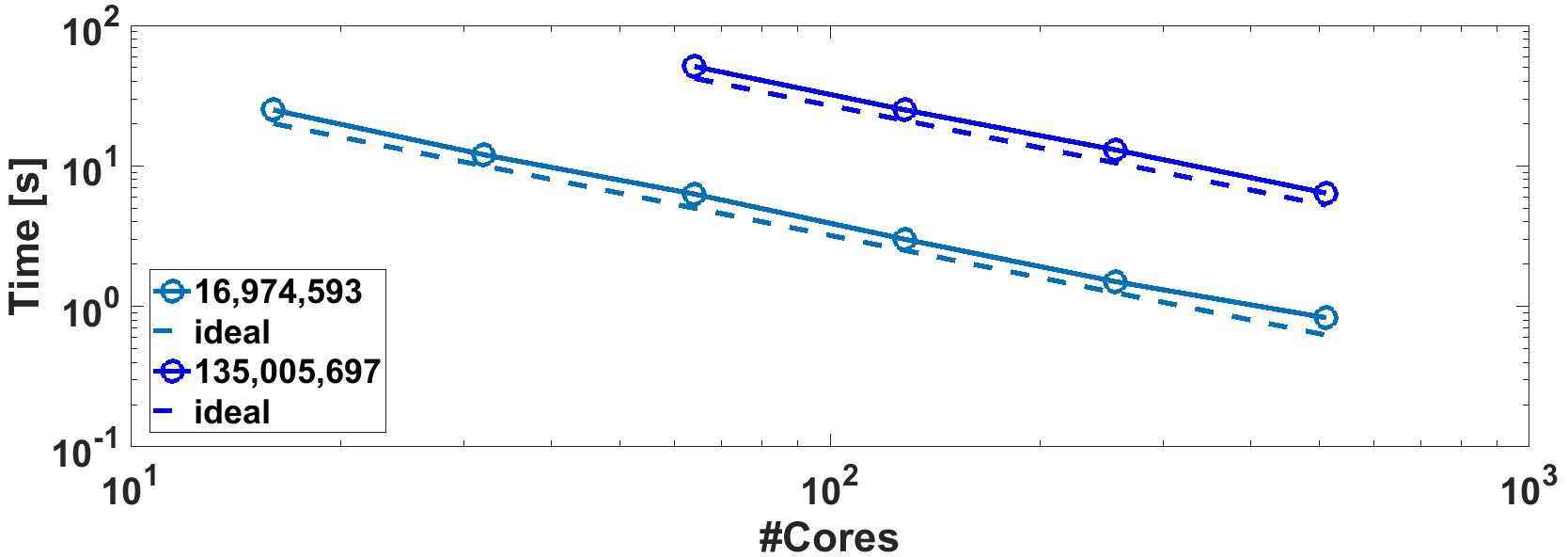}
    \caption{Strong scalability and computational time in seconds (s) with respect to
      the number of cores for uniform refinement and non-nested iterations}
    \label{fig:cpu_nonnested}
\end{figure}
\begin{figure}
  \centering
    \includegraphics[width=1.0\textwidth]{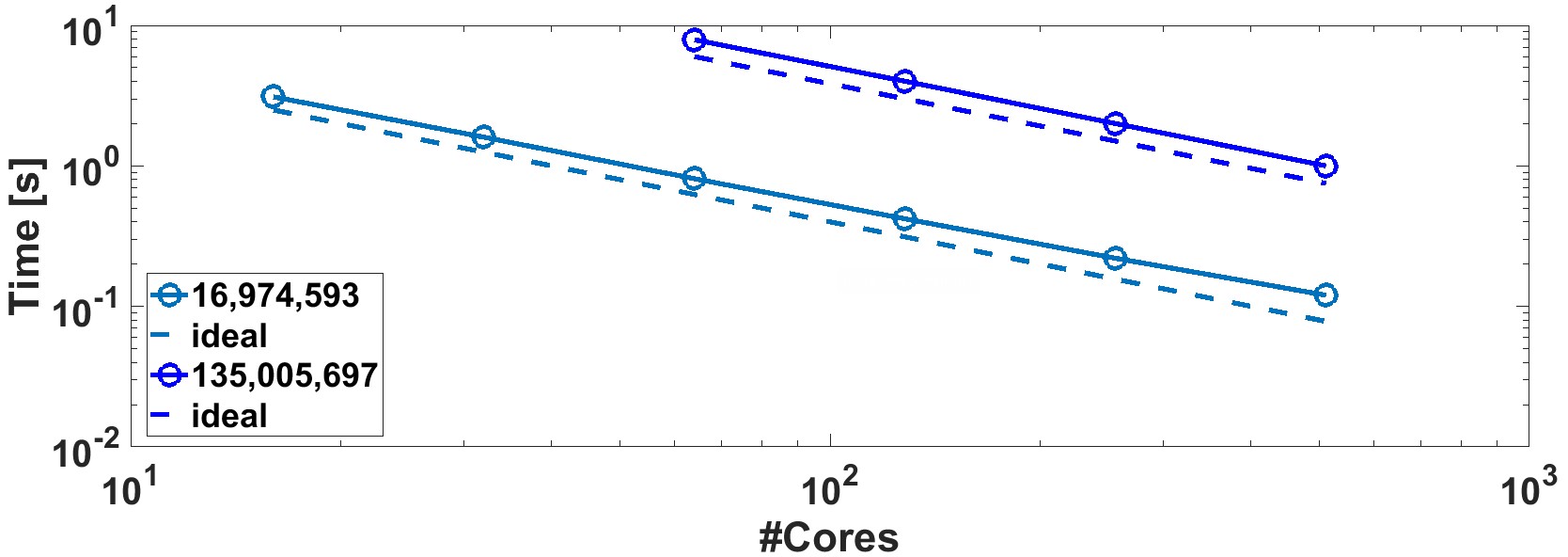}
    \caption{Strong scalability and computational time in seconds (s) with respect to
      the number of cores for uniform refinement and nested iterations}
    \label{fig:cpu_nested}
\end{figure}

\begin{table}
  {\small
    \begin{tabular}{|c|c|c|c|c|c|}
    \hline
    \multirow{2}{*}{\#Dofs}&\multicolumn{5}{c|}{\#Cores}\\  \cline{2-6}
    &16&32&64&128&256\\
    \hline
    $2.76154$e+$6$&$16$ ($1.0$e$-0$)&$16$ ($5.3$e$-1$)&16 ($2.8$e$-1$)&16 ($1.6$e$-1$)&16 ($1.0$e$-1$)\\ \hline
    $1.06728$e+$7$&-&17 ($2.3$e$-0$)&17 ($1.2$e$-0$)&16 ($6.2$e$-1$)&17 ($3.3$e$-1$)\\ \hline
    \end{tabular}
    \caption{Parallel performance on a distributed computer system for adaptive refinement and nested iterations.
    }
    \label{tab:parallel_solver_adaptive_nested}
  }
\end{table}
%
%
\section{Conclusions and Outlook}
\label{LLSY:sec:Conclusions}	
%
%
We provide a rigorous analysis of the discretization error $\|y_d - {\hat y}_{\varrho h}\|_{L_2(\Omega)}$
when replacing the mass matrix $M_h$ arising from the regularization term 
in the reduced optimality system by its lumped version $D_h = \text{lump}(M_h)$.
It turns out that the 
asymptotic behavior of the 
error is not affected by mass lumping when using affine-linear finite elements.
More precisely, we again get the upper bound $c h^s \|y_d\|_{H^s(\Omega)}$, $s\in [0,2]$, for the choice $\varrho = h^4$ 
that provides the optimal balance  between the regularization parameter $\varrho$ 
and the mesh-size $h$.
Moreover, this replacement of $M_h$ by $D_h$ opens 
the way to reduce the discrete reduced optimal optimality system further to a spd 
Schur complement problem that can efficiently be solved by PCG 
since now the matrix-by-vector multiplication is cheap and, surprisingly,
$D_h$ is a  diagonal preconditioner that is spectrally equivalent to the Schur complement $S_h$.
This PCG can 
efficiently
be parallelized as the numerical results show.
These findings provide the perfect ingredients for a nested PCG iteration 
producing 
iterates $y_\ell^{k_\ell}$ that differ from the desired state $y_d$ in the order $O(h_\ell^s)$ 
of the discretization error in asymptotically optimal complexity $O(h_\ell^{-d})$.
The nested iteration process will be stopped when some relative accuracy $\varepsilon \in (0,1)$ 
of the error is reached, or the cost we are willing to pay in terms of 
the control energy density $\|u_L\|^2_{L_2(\Omega)}$ becomes too large.
In this case, we can freeze the regularization (cost) parameter $\varrho_L = h_L^4$,
and continue the nested iteration process with mesh refinement only 
in order to improve to the approximation of the control.

We provide not only numerical results for the case of uniform refinement 
that nicely demonstrated the theoretical predictions 
but also for adaptive refinement when using variable regularization.
The numerical results show that this adaptive approach works well,
but a rigorous numerical analysis is still missing.
Further investigation comprises this analysis, and the generalization 
to larger classes of PDEs like elliptic diffusion-convection-reaction,
parabolic and hyperbolic state equations. 
We refer to  
\cite{LLSY:LangerSteinbachTroeltzschYang:2020SISC,LLSY:LangerSteinbachTroeltzschYang:2020SINUM,LLSY:LangerSteinbachYang:2022arXiv} 
and \cite{LLSY:LoescherSteinbach:2022arXiv:2211.02562} when using space-time fe discretization 
for parabolic and hyperbolic initial-boundary value problems, respectively.
Another future research topic are the consideration of control and state 
(box)
constraints 
in the framework discussed here;
see 
\cite{LLSY:GanglLoescherSteinbach:2023InPreparation}
for first results.
Finally, we mention that singular-perturbed problems 
as
discussed here
also appear  in fluid mechanics where they are known as 
(discrete) differential filter that provide 
approximate deconvolution models of turbulence
\cite{LLSY:BerselliIliescuLayton:2006Book,LLSY:LaytonRebholz:2012:Book,LLSY:John:2016Book}.

	
	\bibliography{LLSY2023MassLumping}

\newcommand{\noopsort}[1]{} \newcommand{\printfirst}[2]{#1}
  \newcommand{\singleletter}[1]{#1} \newcommand{\switchargs}[2]{#2#1}
\begin{thebibliography}{10}

\bibitem{LLSY:AxelssonKaratson:2020NumerMath}
O.~Axelsson and J.~Kar\'{a}tson.
\newblock Superior properties of the {PRESB} preconditioner for operators on
  two-by-two block form with square blocks.
\newblock {\em Numer. Math}, 146(2):335–368, 2020.

\bibitem{LLSY:AxelssonNeytchevaStroem:2018JNM}
O.~Axelsson, M.~Neytcheva, and A.~Str\"om.
\newblock An efficient preconditioning method for state box-constrained optimal
  control problems.
\newblock {\em J. Numer. Math.}, 26(4):185--207, 2018.

\bibitem{LLSY:Bai:2019IMANA}
Z.-Z. Bai.
\newblock Regularized {HSS} iteration methods for stabilized saddle-point
  problems.
\newblock {\em IMA J. Numer. Anal.}, 39(4):1888--1923, 2019.

\bibitem{LLSY:BaiBenzi:2017BIT}
Z.-Z. Bai and M.~Benzi.
\newblock Regularized {HSS} iteration methods for saddle-point linear systems.
\newblock {\em BIT Numer. Math.}, 57(2):287--311, 2017.

\bibitem{LLSY:BaiBenziChenWang:2013IMANA}
Z.-Z. Bai, M.~Benzi, F.~Chen, and Z.-Q. Wang.
\newblock Preconditioned {MHSS} iteration methods for a class of block
  two-by-two linear systems with applications to distributed control problems.
\newblock {\em IMA J. Numer. Anal.}, 33(1):343--369, 2013.

\bibitem{LLSY:BaiPan:2021Book}
Z.-Z. Bai and J.-Y. Pan.
\newblock {\em Matrix Analysis and Computations}.
\newblock SIAM, 2021.

\bibitem{LLSY:BeckerHansbo:2008a}
R.~Becker and P.~Hansbo.
\newblock A simple pressure stabilization method for the {S}tokes equation.
\newblock {\em Comm. Numer. Methods Engrg.}, 24(11):1421--1430, 2008.

\bibitem{LLSY:BenziGolubLiesen:2005ActaNumer}
M.~Benzi, G.~Golub, and J.~Liesen.
\newblock Numerical solution of saddle point problems.
\newblock {\em Acta Numer.}, 14:1--137, 2005.

\bibitem{LLSY:BerselliIliescuLayton:2006Book}
L.~Berselli, T.~Iliescu, and W.~Layton.
\newblock {\em Mathematics of large eddy simulation of turbulent flows}.
\newblock Scientific Computation. Springer-Verlag, Berlin, 2006.

\bibitem{LLSY:BorziSchulz:2009SIAMreview}
A.~Borzi and V.~Schulz.
\newblock Multigrid methods for {PDE} optimization.
\newblock {\em SIAM Review}, 51(2):361--395, 2009.

\bibitem{LLST:Braess:2007a}
D.~Braess.
\newblock {\em Finite Elements: Theory, Fast Solvers, and Applications in Solid
  Mechanics}.
\newblock Cambridge University Press, Cambridge, 2007.

\bibitem{LLSY:ChenZhao:2013JCM}
S.~Chen and J.~Zhao.
\newblock Estimations of the constants in inverse inequalities for finite
  element functions.
\newblock {\em J. Comput. Math.}, 31(5):522--531, 2013.

\bibitem{LLST:DouglasHaaseLanger:2003Book}
C.~Douglas, G.~Haase, and U.~Langer.
\newblock {\em A Tutorial on Elliptic PDE Solvers and Their Parallelization}.
\newblock Software, Environments, and Tools,. SIAM, Philadelphia, 2003.

\bibitem{LLSY:DravinsNeytcheva:2021-003-nc}
I.~Dravins and M.~Neytcheva.
\newblock {\em On the Numerical Solution of State- and Control-constrained
  Optimal Control Problems}.
\newblock Department of Information Technology, Uppsala Universitet, 2021.

\bibitem{LLSY:ElmanSilvesterWathen:2005Book}
H.~C. Elman, D.~J. Silvester, and A.~J. Wathen.
\newblock {\em Finite elements and fast iterative solvers: With applications in
  incompressible fluid dynamics}.
\newblock Numerical Mathematics and Scientific Computation. Oxford University
  Press, New York, 2005.

\bibitem{LLSY:ErnGuermond:2004}
A.~Ern and J.-L. Guermond.
\newblock {\em Theory and Practice of Finite Elements}.
\newblock Springer-Verlag, New York, 2004.

\bibitem{LLSY:GanglLoescherSteinbach:2023InPreparation}
P.~Gangl, R.~L\"oscher, and O.~Steinbach.
\newblock Regularization and finite element error estimates for distributed
  optimal control problems with energy regularization and state or control
  constraints.
\newblock Technical report, TU Graz, 2023.
\newblock in preparation.

\bibitem{LLSY:Hackbusch:1985Book}
W.~Hackbusch.
\newblock {\em Multi-Grid Methods and Applications}.
\newblock Springer Verlag, 1985.

\bibitem{LLSY:HinzePinnauUlbrichUlbrich:2009Book}
M.~Hinze, R.~Pinnau, M.~Ulbrich, and S.~Ulbrich.
\newblock {\em Optimization with PDE Constraints}, volume~23.
\newblock Springer-Verlag, Berlin, 2009.

\bibitem{LLSY:John:2016Book}
V.~John.
\newblock {\em Finite Element Methods for Incompressible Flow Problems},
  volume~51 of {\em Springer Series in Computational Mathematics}.
\newblock Springer, 2016.

\bibitem{LLSY:LangerLoescherSteinbachYang:2022b}
U.~Langer, R.~L\"oscher, O.~Steinbach, and H.~Yang.
\newblock An adaptive finite element method for distributed elliptic optimal
  control problems with variable energy regularization.
\newblock Technical Report arXiv:2209.08811, arXiv.org, 2022.

\bibitem{LLSY:LangerLoescherSteinbachYang:2023LSSC}
U.~Langer, R.~L\"oscher, O.~Steinbach, and H.~Yang.
\newblock Robust iterative solvers for algebraic systems arising from elliptic
  optimal control problems.
\newblock Berichte aus dem Institut f\"ur Angewandte Mathematik 2023/2,
  Technische Universit\"at Graz, Institut f\"ur Angewandte Mathematik, February
  2023.
\newblock submitted to LSSC 2023 Proceedings.

\bibitem{LLSY:LangerLoescherSteinbachYang:2023}
U.~Langer, R.~Löscher, O.~Steinbach, and H.~Yang.
\newblock Robust finite element discretization and solvers for distributed
  elliptic optimal control problems.
\newblock {\em Comput. Meth. Appl. Math.}, 2023.

\bibitem{LLSY:LangerSteinbachTroeltzschYang:2020SINUM}
U.~Langer, O.~Steinbach, F.~Tr{\"o}ltzsch, and H.~Yang.
\newblock Space-time finite element discretization of parabolic optimal control
  problems with energy regularization.
\newblock {\em SIAM J. Numer. Anal.}, 59:675--695, 2021.

\bibitem{LLSY:LangerSteinbachTroeltzschYang:2020SISC}
U.~Langer, O.~Steinbach, F.~Tr{\"o}ltzsch, and H.~Yang.
\newblock Unstructured space-time finite element methods for optimal control of
  parabolic equations.
\newblock {\em SIAM J. Sci. Comput.}, 43:A744--A771, 2021.

\bibitem{LLSY:LangerSteinbachYang:2022b}
U.~Langer, O.~Steinbach, and H.~Yang.
\newblock Robust discretization and solvers for elliptic optimal control
  problems with energy regularization.
\newblock {\em Comput. Meth. Appl. Math.}, 22:97--111, 2022.

\bibitem{LLSY:LangerSteinbachYang:2022arXiv}
U.~Langer, O.~Steinbach, and H.~Yang.
\newblock Robust space-time finite element error estimates for parabolic
  distributed optimal control problems with energy regularization.
\newblock Technical Report arXiv:2206.06455, arXiv.org, 2022.

\bibitem{LLSY:LaytonRebholz:2012:Book}
W.~Layton and L.~Rebholz.
\newblock {\em Approximate deconvolution models of turbulence: Analysis,
  phenomenology and numerical analysis}, volume 2042 of {\em Lecture Notes in
  Mathematics}.
\newblock Springer, Heidelberg, 2012.

\bibitem{LLSY:Lions:1968a}
J.~L. Lions.
\newblock {\em Contr\^{o}le optimal de syst\`{e}mes gouvern\'{e}s par des
  \'{e}quations aux d\'{e}riv\'{e}es partielles}.
\newblock Dunod Gauthier-Villars, Paris, 1968.

\bibitem{LLSY:LoescherSteinbach:2022arXiv:2211.02562}
R.~L\"oscher and O.~Steinbach.
\newblock Space-time finite element methods for distributed optimal control of
  the wave equation.
\newblock Technical Report arXiv:2211.02562, arXiv.org, 2022.

\bibitem{LLSY:MardalWinther:2011NLA}
K.-A. Mardal and R.~Winther.
\newblock Preconditioning discretizations of systems of partial differential
  equations.
\newblock {\em Numer. Linear Algebra Appl.}, 18(1):1--40, 2011.

\bibitem{LLSY:NeumuellerSteinbach:2021a}
M.~Neum\"{u}ller and O.~Steinbach.
\newblock Regularization error estimates for distributed control problems in
  energy spaces.
\newblock {\em Math. Methods Appl. Sci.}, 44(5):4176--4191, 2021.

\bibitem{LLSY:Notay:2019SIMAX}
Y.~Notay.
\newblock Convergence of some iterative methods for symmetric saddle point
  linear systems.
\newblock {\em SIAM J. Matrix Anal. Appl.}, 40(1):122--146, 2019.

\bibitem{LLSY:PearsonStollWathen:2014NLA}
J.~Pearson, M.~Stoll, and A.~Wathen.
\newblock Preconditioners for state-constrained optimal control problems with
  moreau-yosida penalty function.
\newblock {\em Numer. Linear Algebra Appl.}, 21(1):81--97, 2014.

\bibitem{LLSY:PearsonWathen:2012NLA}
J.~Pearson and A.~Wathen.
\newblock A new approximation of the {S}chur complement in preconditioners for
  {PDE}-constrained optimization.
\newblock {\em Numer. Linear Algebra Appl.}, 12(5):816--829, 2012.

\bibitem{LLSY:SchielaUlbrich:2014SIOPT}
A.~Schiela and S.~Ulbrich.
\newblock Operator preconditioning for a class of inequality constrained
  optimal control problems.
\newblock {\em SIAM J. Optim.}, 24(1):435--466, 2014.

\bibitem{LLSY:SchoeberlZulehner:2007SIMAX}
J.~Sch\"oberl and W.~Zulehner.
\newblock Symmetric indefinite preconditioners for saddle point problems with
  applications to {PDE}-constrained optimization problems.
\newblock {\em SIAM J. Matrix Anal. Appl.}, 29:752--773, 2007.

\bibitem{LLSY:SchulzWittum:2008CVS}
V.~Schulz and G.~Wittum.
\newblock Transforming smoothers for pde constrained optimization problems.
\newblock {\em Comput. Visual. Sci.}, 11:207--219, 2008.

\bibitem{LLSY:Steinbach:2008Monograph}
O.~Steinbach.
\newblock {\em Numerical Approximation Methods for Elliptic Boundary Value
  Problems: Finite and Boundary Elements}.
\newblock Springer, New York, 2008.

\bibitem{LLSY:StollWathen:2012NLA}
M.~Stoll and A.~Wathen.
\newblock Preconditioning for partial differential equation constrained
  optimization with control constraints.
\newblock {\em Numer. Linear Algebra Appl.}, 19:53--71, 2012.

\bibitem{LLSY:Troeltzsch:2010a}
F.~Tr\"{o}ltzsch.
\newblock {\em Optimal control of partial differential equations: Theory,
  methods and applications}, volume 112 of {\em Graduate Studies in
  Mathematics}.
\newblock American Mathematical Society, Providence, Rhode Island, 2010.

\bibitem{LLSY:Zulehner:2001MathComp}
W.~Zulehner.
\newblock Analysis of iterative methods for saddle point problems: a unified
  approach.
\newblock {\em Math. Comp.}, 71(238):479--505, 2002.

\bibitem{LLSY:Zulehner:2011SIMAX}
W.~Zulehner.
\newblock Nonstandard norms and robust estimates for saddle point problems.
\newblock {\em SIAM J. Matrix Anal. Appl.}, 32(2):536--560, 2011.

\end{thebibliography}
	\bibliographystyle{abbrv} 
	

\end{document}